\documentclass[psamsfonts]{amsart}

\usepackage{amssymb,amsfonts}
\usepackage[all,arc]{xy}
\usepackage{enumerate}
\usepackage{mathrsfs}
\usepackage{hyperref}
\usepackage{comment}
\usepackage{mathtools}

\newtheorem{thm}{Theorem}[section]
\newtheorem{cor}[thm]{Corollary}
\newtheorem{prop}[thm]{Proposition}
\newtheorem{lem}[thm]{Lemma}

\theoremstyle{definition}
\newtheorem{defn}[thm]{Definition}

\newtheorem{con}[thm]{Condition}

\newtheorem{notn}[thm]{Notation}

\newtheorem{rem}[thm]{Remark}
\newtheorem{rems}[thm]{Remarks}

\theoremstyle{remark}

\newcommand{\C}{\mathbb{C}}
\newcommand{\R}{\mathbb{R}}
\newcommand{\Z}{\mathbb{Z}}

\newcommand{\X}{\mathbf{X}}
\newcommand{\m}{\mu}
\newcommand{\bn}{\mathbf{n}}
\newcommand{\be}{\mathbf{e}}

\DeclareMathOperator{\tr}{tr}
\DeclareMathOperator{\divg}{div}
\DeclareMathOperator{\dist}{dist}

\makeatletter
\let\c@equation\c@thm
\makeatother
\numberwithin{equation}{section}

\title[Uniqueness of high codimension shrinkers and expanders]{Uniqueness of asymptotically conical higher codimension self-shrinkers and self-expanders}

\author{Ilyas Khan}
\address{Ilyas Khan \\ University of Oxford \\ Mathematical Institute \\Andrew Wiles Building\\ Woodstock Rd \\ Oxford, UK \\  OX2 6GG} 
\email{Ilyas.Khan@maths.ox.ac.uk}
\thanks{
The author was partially supported by a USTC Scholarship from the University of Science and Technology of China and the NSF Grants DMS-1147523 and DMS-1510401}

\begin{document} 

\begin{abstract} Let $C$ be an $m$-dimensional cone immersed in $\R^{n+m}$. In this paper, we show that if $F:M^m \rightarrow \R^{n+m}$ is a properly immersed mean curvature flow self-shrinker which is smoothly asymptotic to $C$, then it is unique and converges to $C$ with unit multiplicity. Furthermore, if $F_1$ and $F_2$ are self-expanders that both converge to $C$ smoothly asymptotically and their separation decreases faster than $\rho^{-m-1}e^{-\rho^2/4}$ in the Hausdorff metric, then the images of $F_1$ and $F_2$ coincide. 
\end{abstract}

\maketitle

\tableofcontents

\section{Introduction} 

A proper $m$-dimensional immersion $F:M^m \rightarrow \R^{n+m}$ is called a self-shrinker of the mean curvature flow if it satisfies the following non-linear elliptic equation for every $p \in M$:
\[ H(p) = -\frac{F(p)^\perp}{2} , \]
where $H$ is the mean curvature vector of the immersion and $F(p)^\perp$ is the component of the position vector perpendicular to the tangent plane $T_{F(p)}F(M) \cong F_*(T_p M)$. If $F$ satisfies this equation, then the family of rescalings 
\[F_t: M^m \rightarrow \R^{n+m}, \;\; t \in (0,1]\]
\[F_t(p) = \sqrt{t}F(p) \]
is a solution to the backwards mean curvature flow equation,
\[ -(\partial_t F_t(p))^\perp = H(p,t), \]
where $H(p,t)$ denotes the mean curvature of the immersion $F_t$ at the point $p \in M$. In this paper, we will define the rescaled immersion $\lambda F: M^m \rightarrow \R^{n+m}$ by $(\lambda F)(p) = \lambda F(p) \in \R^{n+m}$ for each $p \in M^m$ and for any $\lambda > 0$. 

Let $\gamma: \Gamma^{m-1} \rightarrow S^{n+m-1} \subset \R^{n+m}$ be a closed $(m-1)$-dimensional properly immersed submanifold of $S^{n+m-1}$. The cone over $\gamma$ is the following immersion into $\R^{n+m}$:
\[ C: \Gamma \times (0, \infty) \rightarrow \R^{n+m}\]
\[C(q, \ell) = \ell \gamma(q) \]
In this paper, we extend the results of Lu Wang in \cite{Wang} and prove the uniqueness of higher codimension self-shrinkers which are properly immersed and locally smoothly asymptotic to a given immersed cone. More precisely, we prove the following theorem:

\begin{thm}\label{maintheorem} Let $C : \Gamma^{m-1} \times (0, \infty) \rightarrow \R^{n+m}$ be a regular cone of dimension $m$ and $R_0$ a positive constant. Suppose that $F: M^m \rightarrow \R^{n+m}$ and $\tilde{F}: \tilde{M}^m \rightarrow \R^{n+m}$ are smooth, connected and proper self-shrinking immersions  into $\R^{n+m} \setminus B_{R_0}$ with boundary contained in $\partial B_{R_0}$. If $F$ and $\tilde{F}$ are smoothly asymptotic to the same cone $C$, then $F$ and $\tilde{F}$ can be reparametrized to converge to $C$ with unit multiplicity and are the same shrinker up to reparametrization. 
\end{thm}

The method of proof runs along similar lines to the proof for embedded hypersurfaces in \cite{Ber} and \cite{Wang}, but numerous complications arise when considering immersions in high codimension. In particular, the major difficulties in the case of immersed shrinkers are the possibility of convergence with multiplicity and the existence of self-intersection points. These considerations motivate a new definition, Definition \ref{convdef}, of local smooth convergence for immersed shrinkers converging with multiplicity to a cone. One can check that in the case of embedded hypersurfaces (which must converge with multiplicity one), Definition \ref{convdef} reduces to the definition of local smooth convergence given on page 4 of \cite{Wang}. 

The first property of Definition \ref{convdef} tells us that for every compact annular region $\bar{B}_{\rho_2} \setminus B_{\rho_1} \subset \R^{n+m}$, the rescaled immersions $\lambda F$ must eventually be caught in some $\epsilon$-band around the truncated cone $C_{\rho_1,\rho_2}$. The second and third properties assume a choice of pullback $G_\lambda$ of $\lambda F$ to the normal bundle $NC_{\rho_1,\rho_2}$ of $C_{\rho_1,\rho_2}$ and describe the convergence of $\lambda F$ in $NC_{\rho_1,\rho_2}$ as a multi-section with constant multiplicity of convergence. Notice that this definition is equivalent to local Euclidean graphical convergence with arbitrary multiplicity over small embedded neighborhoods of the cone $C$. This is a very modest and natural generalization of the notion of convergence used in \cite{Wang}. Additionally, Definition \ref{convdef} is consistent with the definition of asymptotic conicality given in \cite{Berwang} if we consider transversal sections $S$ of the Grassmannian and multi-graphs defined over $S$.

After establishing definitions, we obtain estimates analogous to those found in Lemma 2.1 and 2.2 in \cite{Wang} by considering graphs of vector-valued multifunctions and utilizing the mean curvature flow system in lieu of the mean curvature flow equation. We then attempt to write one self-shrinker as a normal section of the other.  There is not an obvious way to do this--for example, the two shrinkers may converge to $C$ with different multiplicities. However, outside some large radius $R_0$, the annular subsets $F^{-1}(\R^{n+m}\setminus B_R) \subset M$ are homotopic for all $R>R_0$ and evenly cover the cone. As a result, we may appeal to the theory of covering spaces and utilize the smooth lifting property to allow us to ``unwrap" the covering and construct a normal projection to other shrinkers realizing the same covering space. One particularly significant consequence of this argument is Corollary \ref{nomult}, which implies that a shrinker with multiplicity may be written as a normal graph over itself. 

When one shrinker can be represented as a section $V$ of the normal bundle of the other, we generalize the approach of Jacob Bernstein in \cite{Ber} to show that they coincide, instead of the parabolic backwards uniqueness used in \cite{Wang}. We show that \eqref{diffop} holds--that is, this normal section $V$ is an ``almost" eigensection of the drift Laplacian $\Delta^\perp - \frac{1}{2}\nabla^\perp_{F^T}$ on the normal bundle of a shrinker $F$. In Section 4, we modify the results of \cite{Ber} so that they hold for sections of vector bundles with metric connections. In Section 5, we obtain that the normal section $V$ representing the separation between the shrinkers is actually the zero section. In particular, Corollary \ref{nomult} and the fact that $V \equiv 0$ imply that every shrinker asymptotic to a cone may be reparametrized to converge with multiplicity one, significantly simplifying the immersed picture.  This implies that any two shrinkers asymptotic to a cone $C$ in the sense of Defintion \ref{convdef} must cover $C$ equivalently, may be written as normal graphs over one another, and thus coincide. In particular, increasing the multiplicity of the cone in the sense of geometric measure theory does not give rise to new shrinkers.

A proper $m$-dimensional immersion $F:M^m \rightarrow \R^{n+m}$ is called a self-expander of the mean curvature flow if it satisfies the following non-linear elliptic equation for every $p \in M$:
\[ H(p) = \frac{F(p)^\perp}{2} . \]
By a small modification of our arguments for shrinkers, we can also prove a similar theorem for self-expanders with a certain decay rate as in \cite{Ber}. 
\begin{thm}\label{xpandertheorem} Let $C : \Gamma^{m-1} \times (0, \infty) \rightarrow \R^{n+m}$ be a regular cone of dimension $m$ and $R_0$ a positive constant. Suppose that $F: M^m \rightarrow \R^{n+m}$ and $\tilde{F}: \tilde{M}^m \rightarrow \R^{n+m}$ are smooth, connected, proper self-expanding immersions into $\R^{n+m} \setminus B_{R_0}$ with boundary contained in $\partial B_{R_0}$ that are smoothly asymptotic to $C$. If $F$ and $\tilde{F}$ satisfy
\begin{equation}\label{Hausdorffdecay} \lim_{\rho \rightarrow \infty} \rho^{m+1} e^{\frac{\rho^2}{4}} \textrm{dist}_{\mathcal{H}}(F(M) \cap \partial B_\rho)) , \tilde{F}(\tilde M) \cap \partial B_\rho)) = 0,
\end{equation}
then $F$ and $\tilde{F}$ are the same expander up to reparametrization. Here $\textrm{dist}_{\mathcal{H}}$ is Hausdorff distance.
\end{thm}
\begin{rem}
Note that we allow the convergence to be of multiplicity greater than or equal to one.
\end{rem}
Roughly, one may observe that estimates given in Sections 2 and 3 depend only on the relations $|H| \simeq |F^\perp|$, and $|(\partial_t F)^\perp| = |H|$, which hold both for shrinkers flowing backwards in unit time and expanders flowing forwards in unit time. Using these estimates, we obtain a differential inequality for the linearization of the expander equation in Corollary \ref{xpanderlinearization}, which allows us to apply the theory of Bernstein. 

Our results have a number of interesting consequences. A major class of examples of high codimension self-similar mean curvature flow solutions are minimal cones. One application of Theorem \ref{maintheorem} is the following corollary:

\begin{cor}  In any dimension and codimension, the only smooth, complete, properly immersed self-shrinkers asymptotic to a minimal cone are linear subspaces. 
\end{cor}

We similarly obtain a weaker statement for self-expanders from Theorem \ref{xpandertheorem}.

\begin{cor} \label{xpandercor} In any dimension and codimension, if $C$ is a minimal cone, any non-trivial self-expander asymptotic to $C$ may, outside some ball $B_R$, be written as a normal multi-section over $C$ whose magnitude is asymptotically bounded below by a constant multiple of $r^{-m-1}e^{-r^2/4}$. 
\end{cor}

Another application of our results is to Lagrangian mean curvature flow, which has seen many recent articles on the properties of self-shrinkers and self-expanders. Self-expanding Lagrangians, and especially those asymptotic to cones, have been a recent object of interest due to the proposed program of Joyce in \cite{joyce} to resolve the singularities of the almost-calibrated Lagrangian mean curvature flow by gluing in Lagrangian expanders. Indeed, Neves proves in \cite{neves} that a singularity of an almost-calibrated Lagrangian mean curvature flow must be asymptotic to a union of special Lagrangian cones, so Corollary \ref{xpandercor} provides non-trivial information about expanders that may be used to resolve these singularities. 

Additionally, an example from Lagrangian mean curvature flow gives a bound on how much the decay condition in Theorem \ref{xpandertheorem} can be weakened. Anciaux \cite{Anc} and Joyce-Lee-Tsui \cite{jlt} discovered a family of Lagrangian self-expanders asymptotic to transversally intersecting Lagrangian planes, and the uniqueness of these expanders was later proved by Lotay-Neves in \cite{lotayneves} and by Imagi-Joyce-dos Santos in \cite{ijo}. In particular, it was found that in the immersed case, the only two expanders asymptotic to the two transversal Lagrangian planes were the Joyce-Lee-Tsui expander and the planes themselves. This is consistent with our findings, as the decay of the Joyce-Lee-Tsui expander is $O(e^{-r^2/4})$. 

While they do not occur in the almost-calibrated case, Lagrangian self-shrinkers are an important class of singularity models for the general Lagrangian mean curvature flow. A number of recent articles have been written about compact shrinkers in this setting (see \cite{compactness}, \cite{geomtori}, \cite{clifftorus}). Moreover, in \cite{cpq} and \cite{lwcones}, Lee and Wang construct explicit examples of Hamilton stationary self-shrinkers asymptotic to Hamilton stationary Lagrangian cones. By Theorem \ref{maintheorem}, these are unique.

\begin{cor}\label{uniqueLW} Assume that $\lambda_j > 0$ for $1 \le j \le k < n$ and $\lambda_j < 0$ for $k < j \le n$ are integers satisfying $\sum_{j=1}^n \lambda_j >0$. Let
\[ V_C = \{(x_1e^{i\lambda_1 s},\ldots, x_ne^{i\lambda_n s}) \; : \; 0 \le s < \pi, \sum_{j=1}^n \lambda_j x_j^2 = C, (x_1, \ldots, x_n) \in \R^n\},
\]
be the family of Hamilton stationary Lagrangians in $\C^n$ constructed by Lee and Wang. The embedded shrinker $V_{-2\sum_{j=1}^n \lambda_j}$ is the only self-shrinker asymptotic to the cone $V_0$.  
\end{cor}

This article also adds to the existing literature on general mean curvature flow in arbitrary codimension. The properties of higher codimension self-shrinkers have been studied from the perspective of the $\mathcal{F}$-functional (applied in \cite{colmin} to hypersurfaces) in the papers \cite{ArSu}, \cite{andrewswei}, and \cite{leelue}. Higher codimension solitons with the property that the principal normal is parallel have also been studied by Smoczyk \cite{smo} for self-shrinkers and by Kunikawa \cite{kuni} for translating solitons.
 
\section{Preliminaries and Basic Estimates}

When working with immersed sumbanifolds, it will often be convenient to consider Langer charts, which can be roughly thought of as disk-like neighborhoods on the source manifold $M$ of a proper immersion $F:M^m \rightarrow \R^{n+m}$. If $p \in M$, let $A_p$ be an arbitrary affine isometry of $\R^{n+m}$ that takes $F(p)$ to the origin and takes the tangent plane $F_*(T_pM) = T_{F(p)}F(M)$ to the subspace $\R^{m} \times \{0\} \subset \R^{n+m}$. Let $\pi: \R^{n+m} \rightarrow \R^{n+m}$ be the projection to  $\R^{m} \times \{0\} \subset \R^{n+m}$, given in coordinates by
\[(x_1, \ldots, x_m, x_{m+1}, \ldots, x_{n+m}) \mapsto (x_1, \ldots, x_m, 0, \ldots, 0).\]
Let $D^m_r$ denote the $m$-dimensional disk of radius $r$ in $\R^m \times \{0\}$ centered at the origin.

\begin{defn}\label{langer} The Langer chart $U_{p,r} \subset M$ centered at $p$ of radius $r$, is the component of $(\pi \circ A_p \circ F)^{-1}(D^m_r)$ containing $p$. An immersion $F: M^m \rightarrow \R^{n+m}$ is called an $(r,\alpha)$-immersion if for every $q$ in $M$, there exists a function $f_q:D^m_r \rightarrow \R^n$ with $Df_q(0)=0$ and $|Df_q| \le \alpha$ so that the image $(A_q \circ F)(U_{q,r})$ is equal to the graph of $f_q$ over $D^m_r$. In particular, the restriction $F|_{U_{q,r}}$ is an embedding.
\end{defn}

The next proposition gives a quantitative bound on the maximum radius of graphical Langer charts with derivative bounded by $\alpha$. This bound is dependent only on the given $\alpha$ and the magnitude of the second fundamental form $A$ of the immersion $F$.

\begin{prop}\label{ralpha} Let $\alpha >0$. Then for any $C^2$-immersed submanifold $F: M^m \rightarrow \R^{n+m}$ and any $r$ satisfying 
\[ r \le \frac{\alpha}{(1+ \alpha^2)^{\frac{3}{2}}} \frac{1}{\sup_M |A|},
\]
$F$ is an $(r,\alpha)$-immersion. 
\end{prop}

\begin{proof}
See \cite[Lemma 2.1.4]{Coo}.
\end{proof}

We wish to study the self-shrinkers that are smoothly asymptotic to a given cone $C$, so it is necessary to define an appropriate notion of local smooth convergence to make this idea rigorous. 

\begin{defn}\label{convdef}We say that an immersion $F: M^m \rightarrow \R^{n+m}$ is smoothly asymptotic to an immersed cone $C: \Gamma \times (0,\infty) \rightarrow \R^{m+n}$ with multiplicity $k$ (which we will henceforth denote by $kC$) if the following properties hold.

\begin{enumerate} 
\item For any compact subset $K \subset \R^{n+m}\setminus \{0\}$, as $\lambda \rightarrow 0$, the image $\lambda F(M)$ converges to the cone $C(\Gamma \times (0, \infty))$ inside of $K$ in the Hausdorff metric.
\item Let $0 < \rho_1 < \rho_2$. Let $C_{\rho_1, \rho_2}$ denote the annular region $C(\Gamma \times [\rho_1, \rho_2])$, and let $N(\Gamma \times [\rho_1,\rho_2])$ be the normal bundle with fiber metric equal to the pullback metric $C^*g_{\R^{m+n}}$. Let $D^\perp(\Gamma \times [\rho_1,\rho_2])$ be the unit disk subbundle of $N(\Gamma \times [\rho_1,\rho_2])$. Consider the intersection of $\lambda F(M)$ with $\mathcal{T}_\epsilon(C_{\rho_1, \rho_2})$, an $\epsilon$-tubular neighborhood of $C_{\rho_1, \rho_2}$ in $\R^{n+m}$. For sufficiently small $\lambda$, this intersection can be pulled back to a smooth family of immersions $G_\lambda$ from the set $\Sigma := (\lambda F)^{-1}(\mathcal{T}_{\epsilon} (C_{\rho_1, \rho_2})) \subset M$ into the unit disk subbundle $D^\perp(\Gamma \times [\rho_1,\rho_2])$.
\item For every point $p \in \Gamma \times [\rho_1,\rho_2]$, there exists some $r>0$ and a Langer chart $U_{p,r}$ of the cone $C$ such that the intersection $G_\lambda(\Sigma) \cap D^\perp U_{p,r}$ can be parametrized as the image of $k$ sections $\{\sigma^1_\lambda, \ldots, \sigma^k_\lambda\}$ of the unit disk bundle, such that each $\sigma^i_\lambda : U_{p,r} \rightarrow D^\perp U_{p,r}$ converges to the zero section smoothly with respect to $\lambda$ as $\lambda \rightarrow 0$. 
\end{enumerate}
\end{defn}

\begin{rems}\label{rems-smooth-conv-defn} Notice that by condition (2) and the homogeneity of the cone, there exists an $R>0$ and a compact set $K$, such that $F(M)\setminus K$ is an immersed submanifold $G: \Sigma := M \setminus F^{-1}(K) \rightarrow N(\Gamma \times (R, \infty))  $ of the normal bundle of the truncated cone. 

Also note that in property (3), it is not necessary to insist that each neighborhood is covered by exactly $k$ sections if the link $\Gamma$ of $C$ is connected. If $U_{p_1,r}$ and $U_{p_2, r}$ have nonempty intersection and are covered by $k$ and $k'$ local sections respectively, then $k=k'$. Furthermore, the finiteness of the cover, i.e. the condition that $k < \infty$, follows from the properness of the immersion. 

If the link $\Gamma$ is disconnected, then an immersion may in principle converge with distinct multiplicity on each end. In this case, we consider the cone $C$ to be the union of the cones over each connected component of $\Gamma$. Then, we consider separately each end of $M$ asymptotic to each of these individual cones and can assume without loss of generality that Definition \ref{convdef} is satisfied on a cone with connected link.
\end{rems} 

\begin{notn} \label{annularnotn}
Given an immersion $F: M^m \rightarrow \R^{n+m}$ and a compact set $K$ containing the origin, we often consider the annular regions $F^{-1}(\R^{n+m} \setminus K) = M \setminus F^{-1}(K)$. For ease of reading, we introduce the notation $M_K := F^{-1}(\R^{n+m} \setminus K)$. For immersions indexed by $i$ or by $t$, we denote this set by
\[M_{i,K} := F_i^{-1}(\R^{n+m} \setminus K), \; \; M_{t,K} :=F_t^{-1}(\R^{n+m} \setminus K) .
\]
We will often consider $K = B_R$, the ball of radius $R$ containing the origin. For notational simplicity, we set
\[M_R := M_{B_R}, \; \; M_{i, R} := M_{i, B_R}, \;\; M_{t, R} := M_{t, B_R}.
\]
We will also denote the sets $C^{-1}(\R^{n+m}\setminus K)$ and $C^{-1}(\R^{n+m} \setminus B_R)$ by $C_K$ and $C_R$, respectively.

In the case of a cone $C$ with disconnected link (discussed in Remarks \ref{rems-smooth-conv-defn}) we will consider individual connected components (i.e. individual conical ends) of $M_{i,K}$ separately along with their corresponding asymptotic cones with connected link. In the sequel, we conflate the notation for these connected components with that for the entire $M_{i,K}$ and $C$ in the interest of simplicity. This notational convention does not affect the proofs. 
\end{notn}

\begin{lem}\label{curvasymp} Let $F: M^m \rightarrow \R^{n+m}$ be a shrinker smoothly asymptotic to the cone $C : \Gamma^{m-1} \times (0, \infty) \rightarrow \R^{n+m}$, and $F_t$ the corresponding solution to the backwards mean curvature flow. There exist $C_1>0$ and $R_1 >0 $ such that for $p \in M_{t, R_1}$, $t \in (0,1]$ and $0 \le i \le 2$,
\[ |\nabla^i A_{F_t}(p)| \le C_1 |F_t(p)|^{-i-1},\]
where $A_{F_t}$ is the second fundamental form of the immersion $F_t$.
\end{lem}

\begin{proof} The second fundamental form and its covariant derivatives $|\nabla^i A_{C(\Gamma)}|$ of the restricted cone $C: \Gamma \times [1/2, 2] \rightarrow \R^{n+m}$ are bounded, since the link $\Gamma$ is a closed, properly immersed submanifold. For any $\epsilon >0$, for sufficiently small $\lambda>0$, the intersection of $\lambda F(M)$ with the tubular neighborhood $\mathcal{T}_\epsilon (C_{1/2,2})$ can be pulled back to a submanifold $S_\lambda$ of the normal unit disk subbundle $D^\perp (\Gamma \times [1/2,2])$ with boundary contained in $D^\perp (\Gamma \times \{1/2\}) \cup D^\perp (\Gamma \times \{2\})$. By Definition \ref{convdef}, for sufficiently small $\lambda$, the submanifold $S_\lambda$ can be covered by finitely many images of local sections of $D^\perp (\Gamma \times [1/2,2])$ converging smoothly to the zero section as $\lambda \rightarrow 0$. Hence, the second fundamental form $A_{S_\lambda}$ of $S_\lambda$ converges locally smoothly to the second fundamental form $A_{C(\Gamma)}$ on $C_{1/2,2}$ with a uniform rate of convergence and there exists a constant $\lambda_1 >0$ so that there is a $\delta_1>0$ such that if $\lambda \in (0,\delta_1)$, then
\[
|\nabla^iA_{\lambda F}(p)| \le \lambda_1
\]
for $p \in (\lambda F)^{-1}(S_\lambda)$ (by an abuse of notation). Set $R_1 = \frac{2}{\delta_1}$, and choose $p \in M$ such that $|F(p)| > R_1$. Set $\lambda = \frac{1}{|F(p)|} < \frac{\delta_1}{2}$. By scaling, we have
\[
 |\nabla^iA_{F}(p)| = \lambda^{i+1} |\nabla^iA_{\lambda F}(p)| \le \lambda_1|F(p)|^{-i-1},
\]
since $\lambda F(p) \in \partial B_1(0) \subset \R^{n+m}$. If $p \in M_{t, R_1}$ for $t \in (0,1]$, then $|F(p)| > \frac{R_1}{\sqrt{t}} \ge R_1$. By the same scaling argument, setting $\lambda = \frac{1}{|F(p)|} < \frac{\sqrt{t}}{R_1} < \delta_1$, we obtain
\[|\nabla^i A_{F_t}(p)| \le (\sqrt{t})^{-i-1}|\nabla^i A_{F}(p)| \le (\sqrt{t})^{-i-1}|F(p)|^{-i-1}|\nabla^{i}A_{\lambda F}(p)|\le \lambda_1|F_t(p)|^{-i-1},
\]
which completes the proof of the statement.
\end{proof}

We collect some useful facts and notation in the following remark.

\begin{rem}\label{localgraph}
Given a point $\bar{p} = (p, 1)$ in the Riemannian manifold $\Gamma \times (0, \infty)$ with metric $C^*g_{\R^{n+m}}$, we briefly describe local graphical convergence of the family of rescaled shrinkers $\{F_t\}$ to a neighborhood of $\bar{p}$. By Definition \ref{convdef} there exists a time $t_0 > 0$, an $r_0>0$, and a Langer neighborhood $U_{\bar{p},r_0}$ of $\bar{p}$ in the cone such that for $t<t_0$, the intersection of the rescaled shrinker $F_t$ with the unit disk bundle $DU_{\bar{p},r_0}$ is represented by $k$ sections $\{\sigma^1_{\lambda}, \ldots, \sigma^k_{\lambda}\}$, where $\lambda = \sqrt{t}$. To simplify notation, we denote these sections by $\sigma^j_t$ throughout the rest of the paper.  As $t$ approaches 0, the sections $\sigma^j_t$ approach the zero section smoothly uniformly. In particular, the second fundamental forms $A_{\sigma^j_t}(q,\sigma^j_t(q))$ of $\textrm{graph}(\sigma^j_t|_{U_{\bar{p},r_0}})$ will converge uniformly to $A_{U_{\bar{p},r_0}}(q)$. 

Thus, by Proposition \ref{ralpha}, given $\alpha>0$ we can choose $r \in (0,r_0)$ so that $C_{1/2, 2}$ is an $(r,\alpha)$-immersion and each $U_{\bar{p}, r}$ is the graph of a function $f: D^m_r \rightarrow \R^n$. For $t<t_0$,  we may write the images $\sigma^j_t(U_{\bar{p},r}) \cap (\pi \circ A_{\bar{p}} ) ^{-1}(D^m_{r/2})$ as the graphs of a smoothly varying family of time dependent functions $g^j_t : D^m_{r/2} \rightarrow \R^n$ converging to $f: D^m_{r/2} \rightarrow \R^n$. In particular, we may assume that each of the $g^j_t$'s has spatial gradient bounded by $|Dg^j_t| \le 2\alpha$ for all $t<t_0$. 
\end{rem}

The next lemma describes locally-defined functions whose graphs represent the shrinker near a point $(p, \rho) \in \Gamma \times [R_2, \infty)$ (for $R_2 >0$ large) and proves some bounds on their spatial and temporal derivatives. 

\begin{lem}\label{graphest} Let $\alpha > 0$ and $p \in \Gamma$ be given and let the $k$ functions $g^i_t$ be defined for $(p,1) \in \Gamma \times (0,\infty)$ as in Remark \ref{localgraph}. There exists a large radius $R_2>R_1$, and numbers $0 < \epsilon_0 <1$ and $C_2 >0$ depending on $\alpha$, so that if $\bar{p} = (p,\rho) \in \Gamma \times [R_2, \infty)$, then the associated functions 
\[u^i(\cdot,t) := \rho g^i_{\rho^{-2}t} ( \rho^{-1} \cdot) : D^m_{\epsilon_0\rho} \rightarrow \R^n\]
are well defined for $t \in [0,1]$ and satisfy that, 
\[ |D^{j+1}u^i| \le C_2 |\rho|^{-j} \;\;\;\;\;\text{and} \;\;\;\;\;\;\; |D^j \partial_t u^i| \le C_2 |\rho|^{-1 - j} ,\]
for $j=0,1,2$. Here, $D$ and $D^2$ are the Euclidean gradient and Hessian on $\R^m$ respectively, and $\partial_t$ denotes the partial derivative with respect to $t$ fixing points in $\R^m$. 
\end{lem}

\begin{proof} Given $p \in \Gamma$ and $\alpha >0$, let $U_{(p,1), r}$ be a Langer chart in $\Gamma \times (1/2,2)$ satisfying the conditions given in Remark \ref{localgraph}, where $r$ is chosen such that $C_{1/2, 2}$ is an $(r, \alpha)$-immersion. Cover the link $\Gamma \times \{1\}$ by the collection of smaller neighborhoods $\{U_{(p,1), r/2}\}_{p \in \Gamma}$. The $(m-1)$-manifold $\Gamma$ is closed, so we may reduce this cover to a finite subcovering $\mathscr{U} = \{U_{\bar{p}_1, r/2}, \ldots, U_{\bar{p}_M, r/2}\}$. There exists an $\epsilon_0 > 0$ such that for every $p \in \Gamma$, the Langer chart $U_{(p,1),\epsilon_0}$ is contained in an element $U_{\bar{p}_\ell, r/2}$ of $\mathscr{U}$. The functions $\{g^i_t\}$ associated to $U_{(p,1), \epsilon_0}$ can be considered to be restrictions of the functions $\{g^i_t\}$ associated to $U_{\bar{p}_\ell, r/2}$. Thus, we may find a uniform $t_0>0$ such that for all $p \in \Gamma$ the functions $\{g^i_t\}$ associated to the chart $U_{(p,1),\epsilon_0}$ exist for $t<t_0$ and are defined over $D^m_{\epsilon_0}$.

Set $R_2 = \max \{R_1, 4/\sqrt{t_0}\}$, and let $\rho > R_2$. Consider a point $(p, \rho) \in \Gamma \times (R_2, \infty)$. Observe that $\rho^{-2} < t_0/ 16 < t_0$, so the functions $g^i_t: D^m_{\epsilon_0} \rightarrow \R^n$ associated to $U_{(p,1), \epsilon_0}$ are defined for times $t \le \rho^{-2}$. By scaling and the self-similarity of $F_t(x) = \sqrt{t}\cdot F(x)$, setting $\lambda = \sqrt{t} = \rho^{-1}$, we take a rescaling of the immersion $G_\lambda$ into the normal unit disk sub-bundle from Definition \ref{convdef}, 
\[
\rho G_{\rho^{-1}}: (\rho^{-1}F)^{-1}(\mathcal{T}_{\epsilon} (C_{ 1-\delta, 1+\delta})) \rightarrow N(\Gamma \times [\rho (1-\delta), \rho (1+\delta)]).
\]
We obtain this immersion into the $\rho$-disc subbundle of the normal bundle by simply composing $G_{\rho^{-1}}$ with a rescaling of the unit disc bundle by $\rho$. Then, we can see that the intersection of the image of $\rho G_{\rho^{-1}}$ with the cylinder over the disk $D^m_{\rho \epsilon_0}((p, \rho)) \subset T_{(p, \rho)}C$ in the tangent plane to the cone $C$ at $\bar{p} = (p,\rho)$, can be represented as the $k$ graphs of the functions
\[
u^i(\cdot, t) := \rho g^i_{\rho^{-2}t}(\rho^{-1} \cdot) : D^m_{\epsilon_0\rho} \rightarrow \R^n.
\]
Note that the shrinker $F$ can be covered by the images of such graphs by (3) in Definition \ref{convdef} and Remark \ref{localgraph}. We now establish the derivative bounds on the $u^i$'s. We first take the spatial gradient for fixed $t \in [0,1]$. 
\[D u^i(x,t) = D \rho g^i_{\rho^{-2}t}(\rho^{-1}x) = \rho D(g^i_{\rho^{-2}t})( \rho^{-1} x) \circ \rho^{-1}I = D(g^i_{\rho^{-2}t})( \rho^{-1} x)  .\]
However, we know that for $t< t_0$, $|Dg^i_t| \le 2\alpha$ and thus $|Du^i| \le 2 \alpha$. By Lemma \ref{curvasymp}, Lemma \ref{HessBound}, and Lemma \ref{HOT} we know that $|D^2 u^i| \le C_2 \rho^{-1}$ and $|D^3 u^i| \le C_2 \rho^{-2}$, where $C_2$ is a constant depending on $C_1$, $\alpha$, the dimension $m$, and codimension $n$.  

Next, we establish the bounds on $\partial_t u^i$ and its derivatives. For clarity of presentation, we do our calculations for a generic function $u^i$ which we simply denote by $u$. We have the following system for the backwards mean curvature flow of the graph of $u(\cdot,t)$ over $D^m_\frac{\epsilon_0\rho}{2}$. Consider the graph of $u$ given by the embedding $X(w) = (w, u(w,t))$, where $w \in D^m_\frac{\epsilon_0 \rho}{2}$. Then, the embedding satisfies backwards mean curvature flow if
\begin{equation} \label{BMCF} -X_t = \Delta X + a_k \partial_k X,
\end{equation}
where $a_k \partial_k X$ is a vector field generating the appropriate tangential diffeomorphisms. The equation \eqref{BMCF} yields the following system:
\[  0 = - \partial_t w^j = \frac{1}{\sqrt{g}} \frac{\partial}{\partial w^i} \big( g^{ij} \sqrt{g}\big) + a_j, \; \;\;\; j = 1,\ldots, m\]
\[ - \partial_t u^{\alpha} = \frac{1}{\sqrt{g}} \frac{\partial}{\partial w^i} \bigg( g^{ij} \sqrt{g} \frac{\partial u^\alpha}{\partial w^j}\bigg) + a_j \frac{\partial u^\alpha}{\partial w^j}, \;\;\;\; \alpha = 1,\ldots, n\]
By substituting the first system into the second, we obtain the backwards mean curvature flow system for graphs:
\[- \partial_t u^{\alpha} = g^{ij} \frac{\partial^2 u^\alpha}{\partial w^i \partial w^j}, \;\;\;\; \alpha = 1,\ldots, n\]
The metric $g_{ij}$ can be expressed as follows:
\[g_{ij} = \delta_{ij} + \partial_i u \cdot \partial_j u.\]
Since $|Du|$ is uniformly bounded, $g_{ij}$ and $g^{ij}$ are uniformly bounded. Since $|D^2u| \le C_2\rho^{-1}$, we can see immediately that $|\partial_t u| \le C\rho^{-1}$.  Next, we differentiate,
\begin{align*} -\partial_k \partial_t u^\alpha &= (\partial_k g^{ij}) \partial^2_{ij} u^\alpha + g^{ij} \partial^3_{kij} u^\alpha \\
&= g^{hi} g^{lj} \partial_k g_{hl} \partial^2_{ij} u^\alpha + g^{ij} \partial^3_{kij} u^\alpha 
\end{align*}
The derivatives $D^k g^{ij}$ can be expressed as polynomials in $D^{k} g_{ij}$ and lower order derivatives of $g^{ij}$. From this fact and our previous arguments, we deduce that that
\[| g^{hi} g^{lj} \partial_k g_{hl} \partial^2_{ij} u^\alpha + g^{ij} \partial^3_{kij} u^\alpha | \le C(\rho^{-1} \rho^{-1} + \rho^{-2}) \]
and thus $| \partial_k \partial_t u| \le C\rho^{-2}$. A similar calculation yields $| \partial^2_{lk} \partial_t u| \le C\rho^{-3}$. This completes the proof of the lemma. 
\end{proof}

\begin{cor}\label{sepdecay} Let $p \in M$ such that $F(p)$ is contained in the image of a graph $u^i(\cdot, 1): D^m_{\epsilon_0\rho} \rightarrow \R^n$ associated to a point $(q, \rho) \in \Gamma \times [R_2, \infty)$. Then for $y \in D^m_{\epsilon_0\rho}$, 
\[
|D^j u^i(y, 1) - D^j u^i(y,0)| \le C_2\rho^{-1-j},\;\;\;\; j = 0, 1, 2. 
\]
In particular, the distance between $F(p)$ and the image of the cone $C:\Gamma \times [R_2, \infty) \rightarrow \R^{n+m}$ is bounded by $C_2\rho^{-1}$.
\end{cor}

\begin{proof}
We know that for any $y \in D^m_{\epsilon_0 \rho}$, the vector $u^i(y, 0)$ is associated to a point on the image of the cone $C$, and $u^i(y,1)$ is associated to a nearby point on the shrinker $F$, and the displacement between them is given by the vector $u^i(y,1) - u^i(y,0)$. Using Lemma \ref{graphest}, we estimate the magnitude of this displacement.
\[
|u^i(y,1) - u^i(y,0)| = \bigg| \int_0^1 \partial_t u^i(y,t) dt \bigg| \le \int_0^1 |\partial_t u^i(y,t)|dt \le C_2|\rho|^{-1}.
\]
The bounds for the partial derivatives of the $u^i$ are proved in exactly the same way. By assumption, there is some $y \in D^m_{\epsilon_0 \rho}$ associated to $F(p)$, which completes the proof.
\end{proof}

\section{Self-Shrinking Ends as Normal Exponential Graphs} 

In this section, we will prove that if two shrinkers $F_1: M_1 \rightarrow \R^{n+m}$ and $F_2: M_2 \rightarrow \R^{n+m}$ are asymptotic to the same cone $C$ and satisfy a topological condition, then one can be written as a normal graph over the other. More precisely, there exists a radius $R_4 > R_2$ so that there is a compact $K \subset \R^{n+m}$ such that the end $M_{2,K} \subset M_2$ can be pulled back isometrically to a section of the normal bundle of the end $M_{1,R_4}$ (see Notation \ref{annularnotn} for the definition of these submanifolds). We will use a covering space argument to obtain this isometry. First, we will recall some well-known results about smooth covering spaces.

\begin{prop}[Smooth Lifting Criterion]\label{SLC}  Let $\tilde{X}, X,$ and $Y$ be path-connected smooth manifolds. Suppose $p: (\tilde{X}, \tilde{x}_0) \rightarrow (X, x_0)$ is a smooth covering space and $f: (Y, y_0) \rightarrow (X, x_0)$ is a smooth map. Then a smooth lift $\tilde{f} : (Y, y_0) \rightarrow (\tilde{X}, \tilde{x}_0)$ of $f$ exists if and only if $f_* (\pi_1 (Y, y_0)) \subset p_*(\pi_1 (\tilde{X}_0, \tilde{x}_0))$. 
\end{prop}

An immediate consequence of this the uniqueness of smooth covering spaces. 

\begin{prop}\label{uniqcovspace} Let $X, \tilde{X}_1, \tilde{X}_2$ be path-connected smooth manifolds. If $p_1 : \tilde{X}_1 \rightarrow X$ and $p_2 : \tilde{X}_2 \rightarrow X$ are two smooth covering maps, then there exists a diffeomorphism $f: \tilde{X}_1 \rightarrow \tilde{X}_2$ taking a basepoint $\tilde{x}_1 \in p_1^{-1}(x_0)$ to a basepoint $\tilde{x}_2 \in p_2^{-1}(x_0)$ if and only if ${p_1}_*( \pi_1(\tilde{X}_1, \tilde{x}_1)) = {p_2}_*( \pi_1(\tilde{X}_2, \tilde{x}_2)) $.
\end{prop}

Note that given the condition ${p_1}_*( \pi_1(\tilde{X}_1, \tilde{x}_1)) = {p_2}_*( \pi_1(\tilde{X}_2, \tilde{x}_2))$, the map $f: \tilde{X}_1 \rightarrow \tilde{X}_2$ is the smooth lift $\tilde{p}_1$ and its inverse is the lift $\tilde{p}_2$. We will use the convergence of $F_1$ and $F_2$ to $C$ to find appropriate covering maps $p_1:M_{1,R_3} \rightarrow C_{K'}$ and $p_2:M_{2,K} \rightarrow C_{K'}$. Then, if ${p_1}_*( \pi_1(M_{1,R_3} , \tilde{x}_1)) = {p_2}_*( \pi_1(M_{2,K}, \tilde{x}_2))$, the lift $\tilde{p}_1 : M_{1.R_3} \rightarrow M_{2,K}$ will yield a section of the pullback bundle. We then show using our estimates that this section is close by to a normal section.

We first need to show that given a shrinker $F: M \rightarrow \R^{m+n}$ asymptotic to the cone $C$, there is a consistent notion of a ``topology at infinity." More precisely, 

\begin{lem}\label{topinfty}
Let $F:M \rightarrow \R^{m+n}$ be asymptotic to the cone $C$. Then there exists a large radius $R_3>R_2$, such that for any compact sets $K_1, K_2 \supset B_{R_3}$ which are radial with respect to the origin, 
\[
\pi_1(M_{K_1}, x_0) \cong \pi_1(M_{K_2}, x_0),
\]
where $x_0 \in M_{K_1 \cup K_2}$.
\end{lem}
\begin{proof}
Let $r: \R^{m+n} \rightarrow \R$ be the distance function $r(x) = |x|$ on $\R^{n+m}$. If we consider the restriction of this function to the cone $C$, we see that the tangential gradient $\nabla^C r = (Dr)^T = Dr$ (where $Dr$ indicates the Euclidean derivative), and in particular $|\nabla^C r| = |Dr| = 1$. Since the tangent planes of the shrinker $F$ approach those of the cone $C$, there is a radius $R_3 > 0$ such that $||\nabla^F r| - |\nabla^C r|| < \frac{1}{4}$. In particular $|\nabla^F r|$ is non-vanishing and uniformly bounded away from zero. Every point in $M_{R_3}$ is contained in a unique flow line of the negative gradient flow $\dot{x} = - \nabla^F r (x)$ which has velocity bounded above and below. If we consider a compact set $K \supset B_{R_3}$ that is radial with respect to the origin, we can move along the flow lines to homotope $M_K$ to $M_{R_3}$. Observing that the fundamental group is a homotopy invariant concludes the proof of the lemma.
\end{proof}

\begin{lem}\label{covcone}
For $R_3>0$ as in Lemma 3.3, and compact, radial $K \supset B_{R_3}$, the manifold $M_K$ is a $k$-fold covering space of the cone $C_{\tilde{K}}$ where $\tilde{K}$ is a compact, radial subset of $\R^{n+m}$ containing the origin. Furthermore, this covering map is realized in $\R^{n+m}$ as projection along the normal fibers of the cone. 
\end{lem}

\begin{proof}
This follows immediately from property (3) in Definition \ref{convdef}. 
\end{proof}

\begin{lem}\label{pullbacksec}
Let $F_1: M_1 \rightarrow \R^{n+m}$ and $F_2: M_2 \rightarrow \R^{n+m}$ be two self-shrinkers and let $R_3 = R_3(F_1, F_2)$ be the maximum of the radii $R_3>0$ given by Lemma \ref{topinfty} applied to $F_1$ and $F_2$. Let $p_1$ and $p_2$ be the projections of $M_{1, R_3}$ and $M_{2, R_3}$ in the normal bundles of $C_{K_1}$ and $C_{K_2}$, respectively, where $K_1,K_2$ are compact, radial sets containing the origin as in Lemma \ref{covcone}. If $x_0 \in C_{K_1 \cup K_2}$, $p_1(x_1) = p_2 (x_2) = x_0$, and 
\[
{p_1}_*(\pi_1(M_{1,R_3}, x_1)) = {p_2}_*(\pi_1(M_{2, R_3}, x_2)),
\]
then $p_1$ can be lifted to diffeomorphism from $M_{1, 2R_3}$ to $M_{2, \bar{K}}$, where $\bar{K}$ is some compact, radial set containing $B_{R_3}$. Furthermore, this diffeomorphism can be realized as a section of the pullback vector bundle $p_1^{-1}NC$ over $M_{1,2R_3}$.
\end{lem}

\begin{proof}
The first claim is an immediate consequence of the fact that $M_{1,R_3}$ and $M_{2,R_3}$ are covering spaces and the uniqueness of smooth covering spaces, Proposition \ref{uniqcovspace}. To see that this diffeomorphism is a section of the pullback bundle $\pi_1^{-1}NC$, with base $M_{1,2R_3}$ and fibers $N_{p_1(x)} C$ for $x \in M_{1,2R_3}$, we realize the lifted map $\tilde{p}_1 : M_{1,2R_3} \rightarrow M_{2,\bar{K}}$ locally in Euclidean space. Around each point, there is a coordinate patch of the shrinker $F_1$ which by definition can be written as a local section of the normal bundle of the cone. Over this patch of the cone lie $k$ sections representing the intersection of the shrinker $F_2$ with the normal fibers above that patch of the cone. The lift of $p_1$ chooses one of these sections, and thus the map $\tilde{p}_1$ can be represented at a point $x \in M_1$ as $\tilde{p}_1 (x) = x + v$, where $v$ is a vector in $N_{p_1(x)}C$ such that 
\[
v = (p_1(x) - x) + (p_2^{-1}\circ p_1(x) - p_1(x)) = p_2^{-1}\circ p_1(x) -x.
\] 
Note that when restricted to the section determined by the lift, $p_2$ is injective.
\end{proof}

\begin{rem} \label{univcov}We will see in Section 5 that Corollary \ref{nomult} ensures that the lifting criterion is always met in the case of self-shrinkers, possibly after a reparametrization. In the case of self-expanders, the lifting condition in Lemma \ref{pullbacksec} may not be satisfied, i.e. ${p_1}_*(\pi_1(M_{1,R_3}, x_1)) \not = {p_2}_*(\pi_1(M_{2, R_3}, x_2))$. In this case we consider the subgroup
\[G := {p_1}_*(\pi_1(M_{1,R_3}, x_1)) \cap {p_2}_*(\pi_1(M_{2, R_3}, x_2)).\]
Since both ${p_1}_*(\pi_1(M_{1,R_3}, x_1))$ and ${p_2}_*(\pi_1(M_{2, R_3}, x_2))$ have finite index in $\pi_1(C, x_0)$, by basic algebra
\[ [\pi_1(C, x_0): G], [{p_1}_*(\pi_1(M_{1,R_3}, x_1)):G],\; [{p_2}_*(\pi_1(M_{2, R_3}, x_2)): G] < \infty.
\]
Now, we may instead consider the locally isometric covering manifolds $\tilde M_{1,R_3}$ and $\tilde M_{2,R_3}$ of $M_{1,R_3}$ and $M_{2,R_3}$ respectively corresponding to subgroup $G$. Let $\tilde M_{i,R_3}$ be an $\ell$-fold cover of  $M_{i,R_3}$ with projection $\mathscr{P}_i: \tilde{M}_{i,R_3} \rightarrow M_{i,R_3}$, a local isometry. We represent $M_{i, R_3}$ locally as a collection of $k$ local sections $\{\sigma^j\}_1^k$ of the normal bundle over topological disks $U$ contained $C_{K_i}$, as in the proof of Lemma \ref{pullbacksec}. Given a section $\sigma^j$, the inverse image $\mathscr{P}_i^{-1}(\textrm{graph}(\sigma^j))$ consists of $\ell$ disjoint isometric copies of $\textrm{graph}(\sigma^j)$ in $\tilde M_{i,R_3}$. Thus, the local geometry on all relevant scales is unchanged when lifting to the covers $\tilde M_{1,R_3}$ and $\tilde M_{2,R_3}$, so in the rest of the paper, we may replace our expanders by these covers and the isometric immersions $\tilde F_i = F_i \circ \mathscr{P}_i$ whenever appropriate. Note that $\tilde M_{i, R_3}$ is a $k\ell$-fold cover of $C_{K_i}$, where $\ell$ is equal to the index $[{p_1}_*(\pi_1(M_{1,R_3}, x_1)):G]$, for $i = 1,2$. After this ``reparametrization" by a locally isometric cover, all arguments in the paper may be applied without complication to the self-expanders $\tilde{F}_1$ and $\tilde{F_2}$. 
\end{rem}

\begin{rem}\label{rem-orientability}
It may also be noted, continuing the thread of Remark \ref{univcov}, that if $M_{1,R_3}$ is non-orientable, then we may consider the locally isometric orientable double cover $\pi: \widehat{M}_{1,R_3} \rightarrow M_{1, R_3}$. If the section $\sigma: M_{1,2R_3} \rightarrow p_1^{-1}NC$ corresponds to the image of $M_{2,K}$, then we may consider the pullback bundle $\pi^*(p_1^{-1}NC)$ and the pullback section $\pi^*\sigma: \widehat{M}_{1,2R_3} \rightarrow \pi^*(p_1^{-1}NC)$. Just as in Remark \ref{univcov}, the local geometry is unchanged and we may consider the isometric immersions $\tilde{F}_i = F_i \circ \pi$ without loss of generality. This removes the need to consider integration with densities on a non-orientable manifold. 
\end{rem}

\begin{prop}\label{normbund}
Let $F_1$ and $F_2$ satisfy the hypotheses of Lemma \ref{pullbacksec} and let $\epsilon \simeq R_3^{-1}$, where the implicit constant depends only on $n$, $m$, and $C_2 = C_2(F_1, F_2)$, the maximum of the constants given by Lemma \ref{graphest}. Consider the $\epsilon$-tubular neighborhood $\mathcal{T}_\epsilon (M_{1,4R_3})$ of the zero section inside the total space $p_1^{-1} NC$ with the Euclidean pullback metric inherited from $\R^{n+m}$. There exists a compact set $K$ containing $B_{R_3}$ such that the section $\sigma: M_{1,4R_3} \rightarrow p_1^{-1} NC$ corresponding to $M_{2,K}$ is contained in $\mathcal{T}_\epsilon (M_{1,2R_3})$, and normal projection from $M_{2,K}$ to $M_{1, 2R_3}$ with respect to pullback metric on $p_1^{-1}NC$ is well-defined and injective.
\end{prop}

\begin{proof}
By Corollary \ref{sepdecay}, for a sufficiently large $R_3$ and some $K$, the section $\sigma$ representing $M_{2,K}$ is $\epsilon$-close to the zero section of the normal bundle $NM_{1,R_3}$. Thus, around a point $p \in M_{1,4R_3}$ and for some $\epsilon_1 < \epsilon_0$, we can realize the $\epsilon_1 |F_1(p)|$-neighborhood of $p$ in $p_1^{-1}NC$ as a subset of an $\epsilon_1|F_1(p)|$-tubular neighborhood which is itself realized in Euclidean space $\R^{n+m}$ under the exponential map. One may choose $\epsilon$ is sufficiently small that the $\epsilon_1 F_1(p)$-tubular neighborhood contains a connected, embedded neighborhood $U_1$ of $F_1(p)$ and a connected, embedded neighborhood $U_2$ of $F_2(\sigma(p))$. Additionally, by Lemma \ref{graphest} we may represent these embedded pieces as graphs $u_1$ and $u_2$ with bounded gradient over the same $m$-disk of radius $\epsilon_1|F_1(p)|/2$. Lemma \ref{graphest} further tells us that these
\[u_i : D^n_{\frac{\epsilon_1}{2} |F_1(p)|} \rightarrow \R^{n+m} \;\;\;\;\;\; i = 1,2,
\]
are such that 
\[|D^2u_i| \le C |F_1(p)|^{-1}. \]
Then, since $\dist_{\mathcal{T}_\epsilon (M_{1,4R_3})}(\sigma(p),M_{1,4R_3}) < C |F_1(p)|^{-1}$, we may take a minimizer of distance, $q \in M_{1,4R_3} \subset \mathcal{T}_\epsilon(M_{1,4R_3})$. By integrating along paths out of $q$, for sufficiently large $R_3$, we see that the the condition $|D^2u_1| \le C_2 |F_1(p)|^{-1} < C_2R_3^{-1}$ ensures that $M_{1,4R_3}$ only intersects $\bar{B}^{n+m}_{\dist(\sigma(p),q)} (\sigma(p))$ at $q$. This proves that nearest point projection is well-defined.

Now we prove the injectivity of the nearest point projection from $M_{2,K}$ to $M_{1,4R_3}$ in the total space of the fiber bundle $p_1^{-1}NC$. Suppose that there exist two points $q_1$ and $q_2$ in the disk $D^m_{\frac{\epsilon_1}{2} |F_1(p)|}$ such that $F_1(p)$ is the nearest point in $U_1$ to both $\bar{q}_1 = (q_1, u_2(q_1))$ and $\bar{q}_2 = (q_2, u_2(q_2))$. This implies that the vector $\bar{q}_2 - \bar{q}_1$ lies in the normal space $N_{F_1(p)}U_1$. The gradient bound $|Du_1| \le 2\alpha$ implies that the inner product of the unit $n$-blades in $\Lambda^n (\R^{n+m})$ representing $N_{F_1(p)}U_1$ and the normal space $\{0\} \times \R^n$ is bounded below.
\[ |\langle N_{F_1(p)}U_1, \{0\} \times \R^n \rangle| \ge 1 - 2\alpha .
\]
Consequently, the projection of the vector $\bar{q}_2 - \bar{q}_1$ to the normal space $\{0\} \times \R^{n}$, which we denote $(\bar{q}_2 - \bar{q}_1)^\perp$, has magnitude bounded below 
\[
|(\bar{q}_2 - \bar{q}_1)^\perp|\ge (1-2\alpha)|\bar{q}_2 - \bar{q}_1|.
\]
Similarly, we can bound the magnitude of the tangential component, $(\bar{q}_2 -\bar{q}_1)^T = q_2 - q_1$ above.
\[ 
|(\bar{q}_2 - \bar{q}_1)^T| \le 2\sqrt{\alpha}|\bar{q}_2 - \bar{q}_1|.
\] 
Because $(\bar{q}_2 - \bar{q}_1)^\perp = u_2(q_2) - u_2(q_1)$, we can estimate the difference quotient
\[
\frac{|u_2(q_2) - u_2(q_1)|}{|q_2 - q_1|} \ge \frac{1 - 2\alpha}{2\sqrt{\alpha}}.
\]
Let $\mu = (q_2 - q_1) / |q_2 - q_1|$ be the unit vector pointing in the same direction as $q_2 - q_1$. The mean value inequality for vector valued functions tells us that there is a $t_0 \in (0,1)$ such that at $q_3 = (1-t_0)q_1 + t_0 q_2 \in D^m_{\frac{\epsilon_1}{2} |F_1(p)|}$, the directional derivative $D_{\mu} u_2$ has magnitude bounded below:
\[|D_{\mu} u_2(q_3)| \ge \frac{|u_2(q_2) - u_2(q_1)|}{q_2- q_1}
\] 
Thus, $|D u_2(q_3)| \ge (1 - 2\alpha)/2\sqrt{\alpha} > 2\alpha$, for $\alpha$ sufficiently small. This is a contradiction--therefore, normal projection from $M_{2,K}$ to $M_{1,2R_3}$ with respect to pullback metric on $p_1^{-1}NC$ is well-defined and injective inside the tubular neighborhood $\mathcal{T}_\epsilon (M_{1,2R_3}) \subset p_1^{-1}NC$.
\end{proof}

This implies that, perhaps with a slightly modified compact set $K$, $M_{2,K}$ can be written as a section $V$ of the normal bundle of $M_{1,R_4}$, where $R_4> 2R_3$. 

\begin{cor}\label{nomult}
Let $F:M \rightarrow \R^{n+m}$ be a self-shrinker and let the projection map $p: M_{R_3} \rightarrow C_K$ be given by orthogonal projection in the normal bundle of $C_K$. Let $\gamma \in \pi_1(C_K, x_0)$ be a non-trivial deck transformation of $M_{R_3}$, considered as a covering space of $C_K$. Let $F_1 = F$, and define the shrinker $F_2: M_{R_3} \rightarrow \R^{n+m}$ by $x \in M_{R_3} \mapsto F(\gamma(x))$. Let $p_1$ and $p_2$ be the standard normal projections to $C_K$ restricted to the images of $F_1$ and $F_2$ respectively.  The previous two lemmas allow us to write $M_{R_4}$ as a non-trivial section of its own normal bundle. 
\end{cor}

\begin{proof}
The deck transformation permutes the sheets of the $k$-fold covering space without fixed points, so the lift in Lemma \ref{pullbacksec} of the projection map is realized as projection along the normal fibers of $C$ from a sheet $U$ to the sheet $\gamma U$. The statement is immediate from this remark and the previous lemmas.
\end{proof}

In the following remark, we introduce several conventions we will use to calculate derivatives on the various bundles associated to the shrinker.

\begin{rem}\label{rem-bundle-connection-convention} In the remainder of this section, it often will be useful to move between an intrinsically defined vector field on $TM_{1,R}$ or $NM_{1,R}$ and its realization in $\R^{n+m}$. Thus, we recall here the conventions given in \cite[\S 2]{mcfhigher} for the induced connections on the tangent and normal bundles, $TM_{1,R}$ and $NM_{1,R}$. The Levi-Civita connection $\nabla^{F_1}$ on $TM_{1, R}$ induced by the immersion $F_1 : M_{1, R} \rightarrow \R^{n+m}$ is defined by the formula
\[
(F_1)_* \big(\nabla^{F_1}_X Y\big) = \Big(D_{(F_1)_*(X)} \overline{(F_1)_*(Y)}\Big)^T,
\]
where $X,Y \in \Gamma(TM_{1, R})$, $\overline{(F_1)_*(Y)}$ is an arbitrary extension of $(F_1)_*(Y)$ to an open neighborhood in $\R^{n+m}$, the operator $D$ represents the standard differentiation of vector fields in $\R^{n+m}$, and the superscript $\;^T$ denotes projection to $(F_1)_*(TM)$. Similarly, the induced connection $\nabla^\perp$ on the normal bundle $NM_{1,R}$ is given by 
\[
\nabla^\perp_X \nu = \Big(D_{(F_1)_*(X)} \bar{\nu}\Big)^\perp
\]
where $X \in \Gamma(TM_{1, R})$, the vector field $\nu \in \Gamma(NM_{1, R})$ is identified with its realization in the tangent space of $\R^{n+m}$ and arbitrarily extended to an open set by $\bar \nu$, and $\;^\perp$ is projection to the orthogonal complement of $(F_1)_*(TM)$.
\end{rem}

Let $V$ be a section of the normal bundle $NM_{1,R}$ representing the shrinker $F_2$ in the sense of Lemma \ref{normbund}. We now obtain estimates on $V$ and its derivatives.

\begin{lem}\label{secasymp}
For $R_4 > 8R_3$, there exists $C_3 >0$ such that, for the section $V: M_{1,R_4} \rightarrow N M_{1,R_4}$ representing $M_{2,K}$, a point $x \in M_{1,R_4}$, and $0\le i \le 2$,
\begin{equation}\label{Vdecay}
|(\nabla^\perp)^i V(x)| \le C_3 |F_1(x)|^{-1-i}.
\end{equation}
where $\nabla^\perp$ is the induced connection on the normal bundle $NM_{1, R_4}$.
\end{lem}

\begin{proof}
The strategy of the proof is along the same lines as \cite[Lemma 2.3]{Wang}, but written in more general language and modified to accommodate the presence of a non-trivial normal bundle. 

Note that the case $i=0$ is immediate since $|V(x)|$ represents the distance from $x$ to $M_{2,K}$ in the normal bundle $N M_{1,R_4}$ and Corollary \ref{sepdecay} implies that this distance is bounded proportional to $|F_1(x)|^{-1}$. Thus, we begin with the $i=1$ case.

Let $x_0 \in M_{1,R_4}$ and choose a local frame $\{\mathbf{n}_\beta (x) \}_{\beta \in \{1, \ldots, n\}}$ of the normal bundle defined in a neighborhood of $x_0$ which will be made explicit at a later point in the proof.

We consider three points associated to $x_0$. 
\begin{itemize}
\item $z_0 := p_1(x_0)$, where $p_1$ is the standard projection from $M_{1,R_4}$ to the cone $\Gamma \times [R_4, \infty)$.
\item $\zeta_0 := p_2^{-1}(z_0) \in M_{2, K}$, where $p_2^{-1}$ is the local inverse of projection to the cone from $M_{2,K}$ described in the proof of Lemma \ref{pullbacksec}.
\item $y_0 := F_2^{-1}(F_1(x_0) + V^\beta(x_0) \mathbf{n}_\beta(x_0)) \in M_{2, K}$, the nearest point projection from $F_1(x_0)$ to $F_2(M_{2,K})$. 
\end{itemize}

For sufficiently large $R_4 > 0$, we may apply Lemma \ref{graphest} and represent a neighborhood of $F_1(x_0)$ as the graph of a function $u_1$ defined on a domain $\mathcal{D} := D^m_{(\epsilon_0/2)|F_1(x_0)|}(z_0)$ in the tangent space $T_{z_0} C$.  Similarly, a neighborhood of $F_2(\zeta_0)$ can be represented as the graph of a function $u_2$ defined on the same domain $\mathcal{D} \subset T_{z_0} C$. By the $i=0$ case and Corollary \ref{sepdecay}, the distance between $y_0$ and $\zeta_0$ is bounded of order $O(|F_1(x_0)|^{-1})$, so $F_2(y_0)$ is contained in the graph of $u_2$: that is, there exists $q_0 \in T_{z_0} C$ such that $(q_0, u_2(q_0)) = F_2(y_0)$.

We will explicitly parametrize this correspondence between points in the graphs of $u_1$ and $u_2$ so that we can differentiate the nearest point projection from the image of $F_1$ to the image of $F_2$ with respect to the coordinates on the domain $\mathcal{D} \subset T_{z_0}C$. Assume without loss of generality that $T_{z_0}C \subset \R^{n+m}$ coincides with $\R^m \times \{0\}$ and is parametrized by coordinates $(p_1, \ldots, p_m)$ where $z_0$ coincides with $(0,\ldots, 0)$. Let $\{ \be_1,\ldots, \be_{n+m} \}$ be an orthonormal frame for $\R^{n+m}$ such that if $p \in \R^{m} \times \{0\}$, then $p = p_1 \be_1 + \cdots p_m \be_m$. Given $p \in \mathcal{D}$ and $\bar{p} := (p, u_1(p))$ in the image of $F_1$, let $q$ be defined such that $(q, u_2(q)) = \bar{p} + V^\beta(\bar p)\bn_\beta(\bar p)$. That is to say,
\[
(q, u_2(q)) - (p, u_1(p)) = V^\beta(\bar p)\bn_\beta(\bar p).
\]
Another way to see this relation is that $q-p$ is the projection of $V(\bar p)$ to $T_{z_0}C$ and $u_2(q) - u_1(p)$ is the projection of $V(\bar p)$ to the orthogonal complement $(T_{z_0}C)^\perp = \{0\} \times \R^n \subset \R^{n+m}$. This relation between $p$ and $q$ is a system of equations which we restate for clarity.
\begin{equation}\label{eq-tangentpart}
q_h (p) = p_h + \mathbf{e}_h \cdot  \mathbf{n}_\beta(\bar p) V^\beta(\bar p) \;\;\;\; h = 1, \ldots m
\end{equation}
\begin{equation}\label{eq-normalpart}
u^\alpha_2 (q) = u^\alpha_1(p) +  \mathbf{e}_\alpha \cdot \mathbf{n}_\beta(\bar p) V^\beta (\bar p)\;\;\;\; \alpha = 1, \ldots, n.
\end{equation}
The image of $F_1$ in a neighborhood of $F_1(x_0)$ is parametrized by the domain $\mathcal{D}$ via the projection diffeomorphism $ p \xmapsto{F} \bar p = (p, u_1(p))$. Thus, given the vector field $V(\bar p)$ defined on the image of $F_1$ in a neighborhood of $F_1(x_0)$, the following derivatives are equivalent by the chain rule:
\[
\partial_i (V (p,u_1(p))) = D_{F_* \be_i(p)} V (\bar p),
\]
where $\bar p$ is considered as an independent variable and $F_* \be_i(p) = \partial_i F(p) = (\be_i, \partial_i u_1 (p))$. 

We now differentiate \eqref{eq-tangentpart} and \eqref{eq-normalpart} with respect to $p$. 
\begin{equation}\label{eq-deriv1-tangentpart}
\partial_i q_h (p) = \delta_{ih} + \mathbf{e}_h \cdot  \partial_i (V^\beta(\bar p)\mathbf{n}_\beta(\bar p))  \;\;\;\; h = 1, \ldots m
\end{equation}
\begin{equation}\label{eq-deriv1-normalpart}
(\partial_k u^\alpha_2) (q(p)) \partial_i q_k (p) = \partial_i u^\alpha_1(p) +  \mathbf{e}_\alpha \cdot  \partial_i (V^\beta(\bar p) \mathbf{n}_\beta(\bar p))  \;\;\;\; \alpha = 1, \ldots, n.
\end{equation}
Then substitute the equations \eqref{eq-deriv1-tangentpart} into the equations \eqref{eq-deriv1-normalpart} to obtain the following relation. 
\begin{multline}\label{eq-combined-system}(\partial_i u^\alpha_2)(q(p)) +  (\partial_k u_2^\alpha )(q(p))\big(\mathbf{e}_k \cdot  \partial_i (V^\beta(\bar p)\mathbf{n}_\beta(\bar p))\big)  \\ = \partial_i u^\alpha_1(p) + \mathbf{e}_\alpha \cdot  \partial_i (V^\beta(\bar p) \mathbf{n}_\beta(\bar p)). 
\end{multline}

We aim to estimate the magnitude of the derivative $\partial_i V(\bar p)$ at $\bar p = (0, u_1(0)) = x_0$.

\begin{rem}
We have heretofore suppressed the time dependence of the functions $u_j(p, t)$ for $j=1,2$, defined in Lemma \ref{graphest} and only considered the $t=1$ time-slice denoted by $u_j(p) = u_j(p,1)$. In the following calculation, we will use the time dependence of these functions in conjunction with Lemma \ref{graphest} to obtain some necessary estimates.
\end{rem}

We begin by estimating the difference
\begin{multline}\label{eq-three-bounds}
|(\partial_i u^\alpha_2)(q(0)) -  \partial_i u^\alpha_1(0)| \le |(\partial_i u^\alpha_2)(q(0)) -  \partial_i u^\alpha_2(0)| \\ + |(\partial_i u^\alpha_2)(0,1) -  \partial_i u^\alpha_2(0,0)|
 +|(\partial_i u^\alpha_2)(0,0) -  \partial_i u^\alpha_1(0,1)|
\end{multline}

By integration and Lemma \ref{graphest}, the first summand is bounded by
\begin{align*}
    |(\partial_i u^\alpha_2)(q(0)) -  \partial_i u^\alpha_2(0)| &\le |D^2 u_2| |q(0) - 0|\\
    &\le C|F_1(x_0)|^{-1}|y_0 - \zeta_0| \\
    &\le C|F_1(x_0)|^{-2}.
\end{align*}
Note that at time $t=0$, $u_1(p, 0) = u_2(p,0)$ are equal to the graph of the cone $C$ over $T_{z_0}C$, and apply Corollary \ref{sepdecay} to find that
\begin{align*}
    |(\partial_i u^\alpha_2)(0,1) -  \partial_i u^\alpha_2(0,0)| &= \bigg|\int_0^1 \partial_t \partial_i u^\alpha_2 (0,t) dt \bigg|\\
    &\le C|F_1(x_0)|^{-2}.
\end{align*}
Noting again in particular that $u_1(p, 0) = u_2(p,0)$ and applying the same argument yields 
\[ |(\partial_i u^\alpha_2)(0,0) -  \partial_i u^\alpha_1(0,1)| \le C|F_1(x_0)|^{-2}. \]

All in all, this tells us that at $\bar p = (0,u_1(0)) = x_0$
\begin{equation}\label{eq-decay-projections}\mathbf{e}_\alpha \cdot  \partial_i (V^\beta(\bar p) \mathbf{n}_\beta(\bar p)) - (\partial_k u_2^\alpha )(q(0))\big(\mathbf{e}_k \cdot  \partial_i (V^\beta(\bar p)\mathbf{n}_\beta(\bar p))\big)  = O(|F_1(x_0)|^{-2}).
\end{equation}

Intuitively, the first summand is all we care about: we wish to estimate $\nabla^\perp V(x_0)$, and for sufficiently large $R_4$, the normal frame $\{\be_\alpha\}$ is very close to the normal frame $\{\bn_\beta\}$. Thus, the normal part of the total derivative should be captured almost entirely in the first term, and the remainder in the second term should be negligible.

We show this rigorously. First, define the normal frame $\{ \mathbf{n}_\beta(\bar p) \}_{\beta = 1}^n$ by
\begin{equation}\label{def-normal-frame}
    \mathbf{n}_\beta(\bar p) = (\partial_1 u_1^\beta (p), \ldots, \partial_m u_1^\beta (p), 0, \ldots, \underbrace{-1}_{\beta\text{th place}}, \ldots, 0).
\end{equation}
Note that each $\mathbf{n}_\beta(\bar p)$ is perpendicular to the vectors $\partial_i F(p)$, which span the tangent space $T_{\bar p}M_{1, R_4}$. Furthermore, the collection is linearly independent and thus constitutes a frame of the normal space. The metric $h$ on the fibers of $NM_{1,R_4}$ can be expressed with respect to this frame by
\begin{equation}\label{eq-metric-normal-bundle}
    h_{\alpha \beta} = \mathbf{n}_\alpha \cdot \mathbf{n}_\beta = \delta_{\alpha\beta} + Du_1^\alpha \cdot Du_1^\beta.
\end{equation}
By the linear independence of the frame, $h_{\alpha \beta}$ is invertible and its inverse is denoted $h^{\alpha \beta}$.

Next, we expand the left-hand side of \eqref{eq-decay-projections} via the product rule and estimate the summands.
\begin{equation}\label{eq-expanded-normal-part}
(\mathbf{e}_\alpha \cdot \mathbf{n}_\beta(\bar p)) \partial_i V^\beta(\bar p) + (\mathbf{e}_\alpha \cdot  \partial_i \mathbf{n}_\beta(\bar p)) V^\beta(\bar p) 
\end{equation}
\begin{equation}\label{eq-expanded-tangential-part}
(\partial_k u_2^\alpha )(q(0))\big((\mathbf{e}_k \cdot \mathbf{n}_\beta(\bar p))  \partial_i (V^\beta(\bar p)) + (\mathbf{e}_k \cdot  \partial_i \mathbf{n}_\beta(\bar p))V^\beta(\bar p)\big)
\end{equation}

Recall that $F(p) := (p, u_1(p))$, that $\{\partial_i F\}_{i=1, \ldots m}$ are the tangent vectors to the embedded image $F(\mathcal{D})$, and that the induced metric is given by $g_{ij} = \partial_i F \cdot \partial_j F = \delta_{ij} + \partial_i u_1 \cdot \partial_j u_1$. Note that at the origin, the first derivatives $\partial_i u_1(0,0) = 0$, since $T_{z_0}C$ is tangent to the graph of the cone at $z_0$. By integration (c.f. Corollary \ref{sepdecay}) $|\partial_i u_1(0,1)-\partial_i u_1(0,0)| \le C|F_1(x_0)|^{-2}$. Thus, $g_{ij} = \delta_{ij} + O(|F_1(x_0)|^{-2})$ at the origin.  Similarly, the normal bundle metric $h_{\alpha \beta} = \mathbf{n}_\alpha \cdot \mathbf{n}_\beta$ can be written as $h_{\alpha \beta}(x_0) = \delta_{\alpha \beta} + O(|F_1(x_0)|^{-2})$ at the point $x_0$.

The terms involving $\partial_i \bn_\beta$ can be estimated using Lemma \ref{graphest}. We calculate $\partial_i \mathbf{n}_\beta$ explicitly.
\[
\partial_i \mathbf{n}_\beta(\bar p) = (\partial^2_{i1} u_1^\beta (p), \ldots, \partial^2_{im} u_1^\beta (p), 0, \ldots, 0).
\]
Thus, by Lemma \ref{graphest}
\[
|\partial_i \mathbf{n}_\beta(\bar p)| = O(|F_1(x_0)|^{-1}).
\]
Since for any $\gamma = 1, \ldots, n$, 
\[
V^\gamma = h^{\gamma \alpha} \mathbf{n}_\alpha \cdot (V^\beta \mathbf{n}_\beta),
\]
we can estimate
\[
|V^\gamma| \le C\|h^{-1}\||\mathbf{n_\alpha}||V^\beta \mathbf{n}_\beta| \le C\|h^{-1}\||F_1(x_0)|^{-1}.
\]
Recall that at the point $x_0$, the norm $\|h^{-1}\| = 1 + O(|F_1(x_0)|^{-2})$. Thus, the terms in \eqref{eq-expanded-normal-part} and \eqref{eq-expanded-tangential-part} involving products of $\partial_i \mathbf{n}_\beta$ and $V^\beta$ are $O(|F_1(x_0)|^{-2})$ at the point $x_0$.

Thus, we can further simplify \ref{eq-decay-projections} to the following form.
\begin{equation} (\be_\alpha \cdot \bn_\beta - (\partial_k u_2^\alpha)(q(0)) (\be_k \cdot \bn_\beta) )\partial_i V^\beta(\bar p) = O(|F_1(x_0)|^{-2}).
\end{equation}
It is immediate that $\be_\alpha \cdot \bn_\beta = - \delta_{\alpha\beta}$. By the convergence of the first derivatives $Du_1(p, 1)$ to $Du_1(p, 0)$ at the rate $O(|F_1(x)|^{-2})$, for any $\delta > 0$, we can choose $R_4>0$ sufficiently large that $|\be_k \cdot \bn_\beta| < \delta$. Recall that by Lemma \ref{graphest}, $|(\partial_k u_2^\alpha)(q(0))| < C_2$ and choose $\delta \ll C_2$. Thus, if $R_4>0$ is sufficiently large, the matrix $(\be_\alpha \cdot \bn_\beta - (\partial_k u_2^\alpha)(q(0)) (\be_k \cdot \bn_\beta) )_{\alpha \beta}$ is invertible with close to unit operator norm. Thus, we obtain that
\begin{equation*} 
\partial_i V^\beta(\bar p) = O(|F_1(x_0)|^{-2}).
\end{equation*}
Which implies that when $\bar p = (0, u_1(0))$,
\begin{equation}
|\nabla^\perp V(x_0)| \le |\partial_i (V^\beta(\bar p) \bn_\beta(\bar p))| \le C|F_1(x_0)|^{-2}.
\end{equation}
This concludes the proof of the $i=1$ case.

To prove the Lemma for $i = 2$, we begin by differentiating the system given in \eqref{eq-tangentpart} and \eqref{eq-normalpart} twice. This yields the following system:

\begin{equation}\label{eq-deriv2-tangentpart}
\partial_{ij}^2 q_h (p) = \mathbf{e}_h \cdot  \partial_{ij}^2(\mathbf{n}_\beta(\bar p) V^\beta(\bar p)) \;\;\;\; h = 1, \ldots m
\end{equation}

\begin{multline}\label{eq-deriv2-normalpart}
    (\partial^2_{lk} u^\alpha_2) (q(p))\Big(\partial_j q_l (p) \partial_i q_k (p)\Big) + (\partial_k u^\alpha_2) (q(p)) \partial^2_{ij} q_k (p) \\ = \partial^2_{ij} u^\alpha_1(p) +  \mathbf{e}_\alpha \cdot  \partial^2_{ij} (V^\beta(\bar p) \mathbf{n}_\beta(\bar p))  \;\;\;\; \alpha = 1, \ldots, n.
\end{multline}

Substituting \eqref{eq-deriv2-tangentpart} into \eqref{eq-deriv2-normalpart} and rearranging terms, we obtain

\begin{multline}\label{eq-reduced-system}
    \mathbf{e}_\alpha \cdot  \partial^2_{ij} (V^\beta(\bar p) \mathbf{n}_\beta(\bar p)) -  (\partial_k u^\alpha_2) (q(p)) (\mathbf{e}_h \cdot  \partial_{ij}^2(\mathbf{n}_\beta(\bar p) V^\beta(\bar p))) \\ = (\partial^2_{lk} u^\alpha_2) (q(p))\Big(\partial_j q_l (p) \partial_i q_k (p)\Big) - \partial^2_{ij} u^\alpha_1(p)    \;\;\;\; \alpha = 1, \ldots, n.
\end{multline}

Next, use \eqref{eq-deriv1-tangentpart} to expand $\partial_j q_l (p) \partial_i q_k (p)$.

\begin{align*} \partial_j q_l \partial_i q_k &= (\delta_{jl} + \mathbf{e}_l \cdot  \partial_j (V^\beta \mathbf{n}_\beta))(\delta_{ik} + \mathbf{e}_k \cdot  \partial_i (V^\beta\mathbf{n}_\beta))\\
&= \delta_{jl}\delta_{ik} + \mathbf{e}_l \cdot  \partial_j (V^\beta \mathbf{n}_\beta) \delta_{ik} + \mathbf{e}_k \cdot  \partial_i (V^\beta\mathbf{n}_\beta)\delta_{jl} \\
&\qquad + \big(\mathbf{e}_l \cdot  \partial_j (V^\beta \mathbf{n}_\beta)\big) \big(\mathbf{e}_k \cdot  \partial_i (V^\beta\mathbf{n}_\beta)) \big)\\
&= \delta_{jl}\delta_{ik} + O(|F_1(x_0)^{-2}|),
\end{align*}
where the last line comes from the proof of the $i=1$ case. Since $|D^2 u_2| \le C_2 |F_1(x)|^{-1}$ by Lemma \ref{graphest}, 
\begin{equation}
(\partial^2_{lk} u^\alpha_2) (q(0))\Big(\partial_j q_l (0) \partial_i q_k (0)\Big) = (\partial^2_{ij} u^\alpha_2) (q(0)) + O(|F_1(x_0)^{-3}|)
\end{equation}
Evaluating at $x_0$ the right hand side of \eqref{eq-reduced-system} becomes
\[
\partial^2_{ij}u^\alpha_2(q(0)) - \partial^2_{ij}u^\alpha_1(0) + O(|F_1(x_0)^{-3}|).
\]
The difference in the above expression can be estimated using precisely the same application of Lemma \ref{graphest} and Corollary \ref{sepdecay} used to estimate \eqref{eq-three-bounds} in the proof of the $i=1$ case. Thus, the equation \eqref{eq-reduced-system} becomes 
\begin{equation}\label{eq-reduced-equation-2nd-derivs}
\mathbf{e}_\alpha \cdot  \partial^2_{ij} (V^\beta(x_0) \mathbf{n}_\beta(x_0)) -  (\partial_k u^\alpha_2) (q(0)) (\mathbf{e}_h \cdot  \partial_{ij}^2(\mathbf{n}_\beta(x_0) V^\beta(x_0))) = O(|F_1(x_0)^{-3}|)
\end{equation}
Now, expand $\partial^2_{ij} (V^\beta(x_0) \mathbf{n}_\beta(x_0))$.
\begin{align*} \partial^2_{ij} (V^\beta (x_0) \mathbf{n}_\beta(x_0)) &= \partial^2_{ij} V^\beta(x_0) \mathbf{n}_\beta(x_0) + \partial_i V^\beta(x_0) \partial_j \mathbf{n}_\beta(x_0) \\ &\qquad + \partial_j V^\beta(x_0) \partial_i \mathbf{n}_\beta(x_0) + V^\beta(x_0) \partial^2_{ij} \mathbf{n}_\beta(x_0)\\ 
&= \partial^2_{ij} V^\beta(x_0) \mathbf{n}_\beta(x_0) + V^\beta(x_0) \partial^2_{ij} \mathbf{n}_\beta(x_0) +  O(|F_1(x_0)^{-3}|),
\end{align*}
where we estimate $|\partial_j V^\beta(x_0)| =  O(|F_1(x_0)^{-2}|)$ by the proof of the $i=1$ case, and use the bounds on $|\partial_j \mathbf{n}_\beta(x_0)|$ derived in the proof of the $i=1$ case. 

Next, we bound $V^\beta(x_0) \partial^2_{ij} \mathbf{n}_\beta(x_0)$. Calculate $\partial^2_{ij} \mathbf{n}_\beta(x_0)$ explicitly.
\[
\partial^2_{ij} \mathbf{n}_\beta(\bar p) = (\partial^3_{ij1} u_1^\beta (p), \ldots, \partial^3_{ijm} u_1^\beta (p), 0, \ldots, 0).
\]
By Lemma \ref{graphest}, $|\partial^2_{ij} \mathbf{n}_\beta(\bar p)| \le C|D^3u_1| \le C|F_1(x_0)^{-2}|$. Since $|V^\beta| = O(|F_1(x_0)^{-1}|)$, 
\[
|V^\beta(x_0) \partial^2_{ij} \mathbf{n}_\beta(x_0)| = O(|F_1(x_0)^{-3}|).
\]
Thus, equation \eqref{eq-reduced-equation-2nd-derivs} can be further simplified to
\begin{equation}\label{eq-final-expression-2nd-derivs}
\big((\mathbf{e}_\alpha \cdot  \mathbf{n}_\beta(x_0))  -  (\partial_k u^\alpha_2) (q(0)) (\mathbf{e}_h \cdot \mathbf{n}_\beta(x_0) )\big) \partial_{ij}^2 V^\beta(x_0) = O(|F_1(x_0)^{-3}|).
\end{equation}
As in the $i=1$ case, the left-hand side can be inverted at $x_0$ to obtain
\[
|\partial_{ij}^2 V^\beta(x_0)| = O(|F_1(x_0)^{-3}|).
\]
Finally, we prove that the bound on $|(\nabla^\perp)^2 V|$ follows from the previous estimates. Since $|(\nabla^\perp)^2 V| \le |D(\nabla^\perp V)|$, it suffices to estimate $|D(\nabla^\perp V)|$. 
\begin{align}\label{eq-asymp-second-derivs}
   D_j(\nabla^\perp_i V) &= \partial_j\big( \partial_i V^\beta \bn_\beta + h^{\alpha \gamma} V^\beta(\partial_i \bn_\beta \cdot \bn_\alpha) \bn_\gamma \big)  \\
   &= \partial^2_{ij}V^\beta \bn_\beta + \partial_i V^\beta \partial_j \bn_\beta + \partial_j h^{\alpha \gamma} V^\beta(\partial_i \bn_\beta \cdot \bn_\alpha) \bn_\gamma \nonumber \\
   &\qquad + h^{\alpha \gamma} \partial_j V^\beta(\partial_i \bn_\beta \cdot \bn_\alpha) \bn_\gamma  + h^{\alpha \gamma} V^\beta(\partial^2_{ij} \bn_\beta \cdot \bn_\alpha) \bn_\gamma \nonumber \\
   &\qquad + h^{\alpha \gamma} V^\beta(\partial_i \bn_\beta \cdot \partial_j \bn_\alpha) \bn_\gamma + h^{\alpha \gamma} V^\beta(\partial_i \bn_\beta \cdot \bn_\alpha) \partial_j \bn_\gamma \nonumber\\
    &= \partial^2_{ij}V^\beta \bn_\beta + \partial_i V^\beta \partial_j \bn_\beta +   O(|V||F_1(x_0)^{-6}|)   + O(|DV||F_1(x_0)^{-3}|)  \nonumber \\
    &\qquad +    O(|V||F_1(x_0)^{-4}|) +  O(|V||F_1(x_0)^{-2}|) + O(|V||F_1(x_0)^{-4}|) \nonumber \\
   &= O(|F_1(x_0)^{-3}|). \nonumber
\end{align}
This completes the proof of the lemma.
\end{proof}

\begin{lem}\label{linearization}
Let $\sigma$ be a section of the normal bundle $N M_{1,R_4}$. Define a linear operator acting on such sections $\sigma$ at a point $x \in M_{1,R_4}$
\[
\mathcal{L}_0 \sigma = \Delta^\perp \sigma - \frac{1}{2}\nabla^\perp_{F_1(x)^T}\sigma,
\]
where $\Delta^\perp$ and $\nabla^\perp$ are the Laplacian and covariant derivative on the normal bundle respectively. Let $V$ be the section of $N M_{1,R_4}$ given in Lemma \ref{normbund}. There exists $C_4 > 0$ such that at any $x_0 \in M_{1,R_4}$, the following equation is satisfied.
\begin{equation}\label{diffop}
(\mathcal{L}_0 + \frac{1}{2}) V + Q(x_0,V,\nabla^{\perp}V) = \Delta^\perp V - \frac{1}{2}\nabla^\perp_{F_1(x_0)^T}V + \frac{V}{2} + Q(x_0,V,\nabla^{\perp}V) = 0 
\end{equation}
where $F_1(\cdot)^T$ is a vector field of $T M_{1,R_4}$ and the function $Q$ satisfies the inequality
\begin{equation}\label{quadterm}
    |Q(x,V,\nabla^{\perp}V))| \le C_4|F_1(x)|^{-2}(|V|+ |F_1(x)|^{-1}|\nabla^{\perp} V|).
\end{equation}
\end{lem}

\begin{proof} Fix a point $x_0 \in M_{1,R_4}$. We recall from Lemma \ref{secasymp} the parametrization $F: \mathcal{D} \subset \R^m \rightarrow F_1(U_{x_0})$ given by $p \mapsto (p,u_1(p))$, where $U_{x_0}$ is a neighborhood of $x_0$ in $M_{1, R_4}$ and $F_1(U_{x_0})$ is its image in $\R^{n+m}$ (up to affine transformations). If we use the exponential map to obtain a normal coordinate system at $x_0$, we can pull back these coordinates to $\mathcal{D}$ and relabel $(p_1, \ldots, p_n)$ so that $F(0)= x_0$, $\langle \partial_i F(0), \partial_j F(0) \rangle = \delta_{ij}$,  and $\partial^2_{ij} F (0) = A^{\beta}_{ij}(0)\mathbf{n}^\beta$.

We also recall the frame $\{\bn_\beta\}_{\beta = 1}^{n}$ of the normal bundle $NM_{1,R_4}$ given by \eqref{def-normal-frame}. In addition to this frame, we will at times use for convenience an orthonormal geodesic frame of the normal bundle $\{N_\beta\}_{\beta = 1}^n$ in a neighborhood of $x_0$ such that $\nabla^\perp N_\beta(x_0) = 0$ for all $\beta$. 

Let $V= V^\beta\bn_\beta$ be the section of $NM_{1,R_4}$ given in Lemma \ref{normbund}. Since the shrinker $F_2$ can be written as a graph over $F_1$ outside of $B_{R_4} \subset \R^{n+m}$, we have a parametrization of $F_2$ near the point $y_0 = x_0 + V^\beta(x_0) \mathbf{n}_\beta(x_0)$ given by
\[ \tilde{F}: \Omega \rightarrow U \subset F_2^{-1}(\R^{m+n} \setminus K),\;\;\; y_0 \in U
\]
\[\tilde{F}(p) = F(p) + V^\beta(p) \mathbf{n}_\beta(p),\;\;\;\; p \in \Omega
\]
We differentiate in the coordinates $(x_k)_{k=1}^m$ on $\Omega \subset \R^m$ to find the tangent vectors to the shrinker $F_2$. 
\begin{equation}\label{F2tangent}
\partial_i \tilde{F} = \partial_i F + \partial_i V^\beta \mathbf{n}_\beta + V^\beta \partial_i \mathbf{n}_\beta 
\end{equation}
where $\beta=1,\ldots, n$. The second fundamental form of $\tilde{F}$ at $0$ is the normal component of $\partial^2_{ij}\tilde{F}$. To this end, we determine the values of $\partial^2_{ij} \tilde{F}$ and $ (\partial^2_{ij} \tilde{F})^T$.
\begin{align*} \partial_{ij}^2 \tilde{F} &= \partial_{ij}^2 F + \partial_{ij}^2V^\beta \mathbf{n}_\beta + \partial_j V^\beta \partial_i \mathbf{n}_\beta + \partial_i V^\beta \partial_j \bn_\beta  - V^\beta \partial^2_{ij}\bn_\beta \\
&= A_{ij}^\beta \mathbf{n}_\beta + \partial^2_{ij}V^\beta \mathbf{n}_\beta + \partial_j V^\beta \partial_i \mathbf{n}_\beta + \partial_i V^\beta \partial_j \bn_\beta + Q(x_0, V, DV),
\end{align*}
such that $|Q(x_0, V, DV)| \le C_4 |F_1(x_0)|^{-2}(|V| + |F_1(x_0)|^{-1}|DV|)$. In the last line, we used the asymptotics from Lemmas \ref{secasymp} and \ref{curvasymp} to conclude that $|\partial^2_{ij}\bn_\beta| \le C|D^3 u_1|$ and thus $|V^\beta \partial^2_{ij}\bn_\beta| \le C|V||F_1(x_0)|^{-2}$. 

We now calculate the metric tensor with respect to the coordinate chart $\tilde{F}$. Here, we use the geodesic normal frame $\{N_\alpha\}$ to achieve a simpler and more geometrically meaningful expression. Note that $\partial_j N_\alpha = A^\alpha_{jl} \partial_l F$ at $x_0$ by the vanishing of the Christoffel symbols. 
\begin{align*}\tilde{g}_{ij} &= \partial_i \tilde{F} \cdot \partial_j \tilde{F} \\
&=(\partial_i F + \partial_i V^\beta N_\beta - V^\beta A^\beta_{ik} \partial_k F)\cdot(\partial_j F + \partial_j V^\alpha N_\alpha - V^\alpha A^\alpha_{jl} \partial_l F) \\
&= \delta_{ij} + \partial_iV^\beta \partial_j V^\beta - V^\beta A^\beta_{jl}\delta_{il} - V^\beta A^\beta_{ik}\delta_{jk} + V^\beta V^\alpha A^\alpha_{ik}A^\beta_{jl}\delta_{kl} \\
&= \delta_{ij} + \langle \partial_i V, \partial_j V \rangle - 2 \langle V, A_{ij} \rangle + Q_{ij}(x_0, V, DV),
\end{align*}
where $|Q_{ij}(x_0, V, DV)| \le C_4 |F_1(x_0)|^{-3}(|V| + |DV|)$. We now calculate the inverse $(\tilde{g}^{ij})$ up to quadratic error. This reduces to inverting $I + (\langle \partial_i V, \partial_j V \rangle - 2\langle V, A_{ij}\rangle)_{ij}$, where $S = (\langle \partial_i V, \partial_j V \rangle - 2\langle V, A_{ij}\rangle)_{ij}$ is a symmetric matrix by the symmetry of the second fundamental form. Let 
\[ 
p(x) = \sum_{k=0}^m a_k x^k = x^m + \tr (S) x^{m-1} + \cdots + \det (S)
\]
be the characteristic polynomial of $S$. The inverse $(I + S)^{-1}$ has the following form:
\[
(I + S)^{-1} =\bigg( \sum_{j=0}^m (-1)^j a_j\bigg)^{-1}\bigg(-\sum_{k=0}^{m-1} \bigg[ \sum_{j=0}^{m-1-k} (-1)^j a_{j+k+1} \bigg] S^k \bigg)
\]
We use the fact that $S^k = (Q_{ij}(x_0, V, DV))_{ij}$ for $k \ge 2$ and $a_j = Q(x_0, V, DV)$ for $j \le m-2$, to obtain
\begin{align*}
(I + S)^{-1} &= \frac{ \big( ((-1)^m + (-1)^{m-1} \tr(S))I + ((-1)^{m-1} + (-1)^{m-2} \tr(S))S \big)}{\big( (-1)^m + (-1)^{m-1} \tr(S) \big)}  \\
& \hspace{6cm} + (Q_{ij}(x_0, V, DV))_{ij} \\
&= I - S + (Q_{ij}(x_0, V, DV))_{ij}
\end{align*}
Thus, we obtain $\tilde{g}^{ij} = \delta^{ij} - \langle \partial_i V, \partial_j V \rangle + 2\langle V, A_{ij} \rangle + Q_{ij}(x_0, V, DV)$. Now we can calculate the tangential components of $\partial_{ij}^2 \tilde{F}$.
\begin{align*}
(\partial_{ij}^2 \tilde{F})^T &= \tilde{g}^{lm}(\partial^2_{ij} \tilde{F} \cdot \partial_l \tilde{F}) \partial_m \tilde{F} \\ 
&= \tilde{g}^{lm}(((A_{ij}^\beta  + \partial^2_{ij}V^\beta )\mathbf{n}_\beta +\partial_j V^\beta \partial_i \mathbf{n}_\beta + \partial_i V^\beta \partial_j \bn_\beta )\cdot \partial_l \tilde{F}) \partial_m \tilde{F} \\
& \hspace{6cm} + Q(x_0, V, DV) \\
&= \tilde{g}^{lm}\big( (A_{ij}^\beta  + \partial^2_{ij}V^\beta ) \bn_\beta \cdot \partial_l (V^\alpha \bn_\alpha) + (\partial_j V^\beta \partial_i \mathbf{n}_\beta + \partial_i V^\beta \partial_j \bn_\beta) \cdot \partial_l F \big) \partial_m \tilde{F}\\
&\qquad + Q(x_0, V, DV) \\
&=\big( A_{ij}^\beta\partial_l V^\beta + (\partial_j V^\beta \partial_i \mathbf{n}_\beta + \partial_i V^\beta \partial_j \bn_\beta)  \cdot \partial_l F \big) \partial_l \tilde{F} + Q(x_0, V, DV) \\
&=A_{ij}^\beta\partial_l V^\beta \partial_l F + \partial_j V^\beta \partial_i \mathbf{n}_\beta + \partial_i V^\beta \partial_j \bn_\beta + Q(x_0, V, DV)
\end{align*}
Throughout the above computation, we are moving terms into the quadratic error term using the decay estimates from Lemma \ref{secasymp}. Note that the last line comes from observing the following bound on the normal part of $\partial_j V^\beta \partial_i \mathbf{n}_\beta + \partial_i V^\beta \partial_j \bn_\beta$ at $x_0$:
\begin{align*}
|(\partial_j V^\beta \partial_i \mathbf{n}_\beta + \partial_i V^\beta \partial_j \bn_\beta)^\perp| &\le C|\partial_i V^\beta||\partial_i \mathbf{n}_{\beta} \cdot \bn_\alpha| \\
& \le C|\partial_i V^\beta||D^2u_1(x_0)||Du_1(x_0)| \\
& \le C|DV||F_1(x_0)|^{-3}.
\end{align*}
Putting everything together, we find the normal component of $\partial_{ij}^2 \tilde{F}$:
\begin{align}\label{dderivnorm}
(\partial_{ij}^2 \tilde{F} )^\perp &= \partial_{ij}^2 \tilde{F} - (\partial_{ij}^2 \tilde{F})^T \nonumber \\
&=A_{ij}^\beta \mathbf{n}_\beta + \partial^2_{ij}V^\beta \mathbf{n}_\beta + \partial_j V^\beta \partial_i \mathbf{n}_\beta + \partial_i V^\beta \partial_j \bn_\beta  \nonumber \\
& \qquad - \big( A_{ij}^\beta\partial_l V^\beta \partial_l F + \partial_j V^\beta \partial_i \mathbf{n}_\beta + \partial_i V^\beta \partial_j \bn_\beta\big)  + Q(x_0, V, DV) \nonumber \\
&= A_{ij}^\beta \mathbf{n}_\beta + \partial^2_{ij}V^\beta \mathbf{n}_\beta - A_{ij}^\beta\partial_l V^\beta \partial_l F + Q(x_0, V, DV)
\end{align}

Next we calculate the normal component of the position vector $\tilde{F}$. We use the normal frame $\{N_\alpha\}$ for these computations. First find the tangential component
\begin{align*}
\tilde{F}^T = \tilde{g}^{kl}(\tilde{F} \cdot \partial_k \tilde{F}) \partial_l \tilde{F} &= \tilde{g}^{kl}\big( (F + V^\beta N_\beta) \cdot (\partial_k F + \partial_k V^\alpha N_\alpha - V^\alpha A^\alpha_{km}\partial_m F)\big) \partial_l \tilde{F} \\
&= \tilde{g}^{kl}\big(F^k + \partial_k V^\alpha F^\alpha + V^\alpha \partial_k V^\alpha - V^\alpha A^\alpha_{km}F^m \big) \partial_l \tilde{F} \\
&= \big(F^l + \partial_l V^\alpha F^\alpha  - V^\alpha A^\alpha_{lm}F^m \big) \partial_l \tilde{F} \\
&\qquad - \langle \partial_l V, \partial_k V \rangle(F^k + \partial_k V^\alpha F^\alpha - V^\alpha A^\alpha_{km}F^m \big) \partial_l \tilde{F}\\
&\qquad + 2V^\beta A^\beta_{kl}(F^k + \partial_k V^\alpha F^\alpha - V^\alpha A^\alpha_{km}F^m \big) \partial_l \tilde{F} \\
& \hspace{2cm} +Q(x_0, V, DV) \\
&= F^l \partial_l F + F^l \partial_l V^\beta N_\beta -  V^\beta A^\beta_{lm}F^l\partial_m F + \partial_l V^\beta F^\beta \partial_l F \\
&\qquad - V^\beta A^\beta_{lm} F^m \partial_l F - \langle \partial_l V, \partial_k V \rangle F^k \partial_l F + 2V^\beta A^\beta_{kl}F^k \partial_l F \\
&\qquad + Q(x_0, V, DV)
\\
&= (F^l + \partial_l V^\beta F^\beta - \langle \partial_l V, \partial_k V \rangle F^k )\partial_l F + F^l \partial_l V^\beta N_\beta \\
&\qquad + Q(x_0, V, DV) 
\end{align*}
Note that in the fourth equality, the fact that $|F^\perp| = 2|H|$ is used to absorb $\partial_\ell V^\alpha F^\alpha \partial_\ell V^\beta N_\beta$ into $Q(x_0, V, DV)$. Thus,
\begin{align}\label{posnorm}
\tilde{F}^\perp &= \tilde{F} - \tilde{F}^T \nonumber \\
&= F^\perp + V^\beta N_\beta - (\partial_l V^\beta F^\beta - \langle \partial_l V, \partial_k V \rangle F^k)\partial_l F - F^l \partial_l V^\beta N_\beta +Q(x_0, V, DV) 
\end{align}
Since $\tilde{F}$ is a self-shrinker, we can substitute \eqref{dderivnorm} and \eqref{posnorm} into the self-shrinker equation.
\begin{align}\label{selfshrinkeqn}
H_{\tilde{F}} &= \tilde{g}^{ij} (\partial^2_{ij} \tilde{F})^\perp = -\frac{1}{2} \tilde{F}^\perp
\end{align}
\begin{multline*}\tilde{g}^{ij}(A_{ij}^\beta \mathbf{n}_\beta + \partial^2_{ij}V^\beta \mathbf{n}_\beta - A_{ij}^\beta\partial_l V^\beta \partial_l F) \\ = -\frac{1}{2} \big( F^\perp + V^\beta N_\beta - (\partial_l V^\beta F^\beta - \langle \partial_l V, \partial_k V \rangle F^k)\partial_l F - F^l \partial_l V^\beta N_\beta\big) \\ + Q(x_0, V, DV)  
\end{multline*}
Note that $\tilde g^{ij}$ is $\delta^{ij} + O(|V||F_1(x_0)|^{-1})$, so we can trace with respect to $\delta_{ij}$ and place the remainder terms in $Q(x_0, V, DV)$. 
\begin{multline*}H_F^\beta \mathbf{n}_\beta + \partial^2_{ii}V^\beta \mathbf{n}_\beta - H_F^\beta \partial_l V^\beta \partial_l F \\ = -\frac{1}{2}F^N - \frac{1}{2} V^\beta N_\beta  + \frac{1}{2} (\partial_l V^\beta F^\beta - \langle \partial_l V, \partial_k V \rangle F^k)\partial_l F  + \frac{1}{2} F^l \partial_l V^\beta N_\beta + Q(x_0, V, DV)  
\end{multline*}
By the calculation \eqref{eq-asymp-second-derivs} in Lemma \ref{secasymp}, make the substitution
\[
\partial^2_{ij} V^\beta \bn_\beta = D_j(\nabla_i^\perp V) - \partial_i V^\beta \partial_j \bn_\beta + Q(x_0, V, DV).
\]
Tracing and taking the normal part of our equation, we obtain 
\begin{multline*}
\big(H_F^\beta \mathbf{n}_\beta + D_i(\nabla_i^\perp V) - \partial_i V^\beta \partial_j \bn_\beta - H_F^\beta \partial_l V^\beta \partial_l F \big)^{\perp}  \\
\qquad = -\frac{1}{2}\Big( F^\perp + V^\beta N_\beta - (\partial_l V^\beta F^\beta - \langle \partial_l V, \partial_k V \rangle F^k)\partial_l F - F^l \partial_l V^\beta N_\beta + Q(x_0, V, DV) \Big)^\perp
\end{multline*}
As $F$ satisfies the self-shrinker equation, $H_F^\beta \mathbf{n}_\beta$ and $-\frac{1}{2}F^\perp$ cancel. Note that the normal part of $\partial_i V^\beta \partial_j \bn_\beta$ at $x_0$ is $O(|DV||F_1(x_0)|^{-3})$ and this term can thus be moved into $Q(x_0, V, DV)$. Note also that $Q(x_0, V, DV) = Q(x_0, V, \nabla^\perp V)$, as $|DV| \lesssim |\nabla^\perp V| + |V|$, with universal implicit constant. Removing all tangential parts, we obtain 

\[(\nabla^\perp)^2_{ii} V = -\frac{1}{2}V^\beta N_\beta + \frac{1}{2} F^l \partial_l V^\beta N_\beta + Q(x_0, V, \nabla^{\perp}V)  \]

We can rewrite this expression as 
\[
\Delta^\perp V - \frac{1}{2}\nabla^\perp_{F^T} V + \frac{V}{2} + Q(x_0, V, \nabla^\perp V) = 0, 
\]
where we are evaluating at $x_0$. This concludes the proof of the lemma. 
\end{proof} 

We now prove that the analogous equation holds for self-expanders.

\begin{cor}\label{xpanderlinearization}
Suppose that $F_1$ and $F_2$ are self-expanders asymptotic to the same asymptotic cone $C$. All the previous conclusions of Sections 2 and 3 for self-shrinkers up to but not including Lemma \ref{linearization} apply to expanders $F_1$ and $F_2$. Let $\sigma$ be a section of the normal bundle $N M_{1,R_4}$. Define a linear operator acting on such sections $\sigma$ at a point $x \in M_{1,R_4}$
\[
\mathcal{L}_0^+ \sigma = \Delta^\perp \sigma + \frac{1}{2}\nabla^\perp_{F_1(x)^T}\sigma,
\]
where $\Delta^\perp$ and $\nabla^\perp$ are the Laplacian and covariant derivative on the normal bundle respectively. Let $V$ be the section of $N M_{1,R_4}$ given in Lemma \ref{normbund}. There exists $C_4 > 0$ such that at any $x_0 \in M_{1,R_4}$, the following equation is satisfied.
\begin{equation}\label{diffopxp}
(\mathcal{L}_0^+ - \frac{1}{2}) V + Q(x_0,V,\nabla^{\perp}V) = \Delta^\perp V + \frac{1}{2}\nabla^\perp_{F_1(x_0)^T}V - \frac{V}{2} + Q(x_0,V,\nabla^{\perp}V) = 0 
\end{equation}
where $F_1(\cdot)^T$ is a vector field of $T M_{1, R_4}$ and the function $Q$ satisfies the inequality
\begin{equation}\label{quadtermxp}
    |Q(x,V,\nabla^{\perp}V))| \le C_4|F_1(x)|^{-2}(|V|+ |F_1(x)|^{-1}|\nabla^{\perp} V|).
\end{equation}
\end{cor}

\begin{proof}
To see that the properties for shrinkers proved in Sections 2 and 3 apply also to expanders, notice that all the estimates in Section 2 depend only on the relations $|H| \simeq |F^\perp|$ and $|(\partial_t F)^\perp| = |H|$, which apply equally to shrinkers and expanders. To prove the rest of the corollary, we follow the proof of Lemma \ref{linearization} until \eqref{selfshrinkeqn} which is the first time the self-shrinker equation is explicitly used. At this step, we instead plug in the self-expander equation:
\begin{align*}
H_{\tilde{F}} &= \tilde{g}^{ij} (\partial^2_{ij} \tilde{F})^\perp = \frac{1}{2} \tilde{F}^\perp
\end{align*}
This becomes
\begin{multline*}\tilde{g}^{ij}(A_{ij}^\beta \mathbf{n}_\beta + \partial^2_{ij}V^\beta \mathbf{n}_\beta - A_{ij}^\beta\partial_l V^\beta \partial_l F) \\ = \frac{1}{2} \big( F^\perp + V^\beta N_\beta - (\partial_l V^\beta F^\beta - \langle \partial_l V, \partial_k V \rangle F^k)\partial_l F - F^l \partial_l V^\beta N_\beta\big) \\ + Q(x_0, V, DV)  
\end{multline*}
Arguing as in Lemma \ref{linearization}, obtain
\[
(\nabla^\perp)^2_{ii} V = \frac{1}{2}V^\beta N_\beta - \frac{1}{2} F^l \partial_l V^\beta N_\beta + Q(x_0, V, \nabla^{\perp}V),  
\]
which can be rewritten as
\[
\Delta^\perp V + \frac{1}{2}\nabla^\perp_{F_1(x_0)^T}V - \frac{V}{2} + Q(x_0,V,\nabla^{\perp}V) = 0
\]
evaluated at $x_0$. This completes the proof of the corollary.
\end{proof}

\section{Unique Continuation on Weakly Conical Ends}

In the second half of the paper, we  prove a unique continuation result for higher codimension self-shrinking ends asymptotic to a cone. We do this by extending to the vector-valued case a recent result \cite{Ber} of Jacob Bernstein on the asymptotic structure of almost eigenfunctions of drift Laplacians on conical ends. 
\\
\indent A weakly conical end is a triple $(\Sigma, g, r)$ consisting of a smooth $m$-dimensional manifold,
$\Sigma$, with $m \ge 2$, a $C^1$-Riemannian metric, g and a proper unbounded $C^2$ function $r: \Sigma \rightarrow (R_\Sigma, \infty)$ where $R_\Sigma \ge 1$, and such that there is a constant $\Lambda \ge 0$ with the property that 
\begin{equation}\label{radderiv}
||\nabla_g r| - 1| \le \frac{\Lambda}{r^4} \le \frac{1}{2} \text{ and } |\nabla^2_gr^2 - 2g| \le \frac{\Lambda}{r^2} \le \frac{1}{2}.
\end{equation}
We consider the tuple $(B, p, \Sigma, h, \nabla)$, which represents a vector bundle $B$ over $\Sigma$ with metric $h$, compatible metric connection $\nabla$, and projection $p$. We furthermore assume that the connection $\nabla$ satisfies the following condition.

\begin{con}\label{cond-curvature-decay}
The curvature $R^\nabla: \Gamma(T\Sigma) \times \Gamma(T\Sigma) \times \Gamma(B) \rightarrow \Gamma(B)$ of the connection $\nabla$ is defined and satisfies
\begin{equation}\label{eq-radial-curvature-decay-bundle}
    R^\nabla (X,r\nabla_g r)W = O(|X||W|r^{-3}),
\end{equation}
for $r \ge R_\Sigma$, and any vector field $X \in \Gamma(T\Sigma)$ and section $W \in \Gamma(B)$.
\end{con}

In particular, Condition \ref{cond-curvature-decay} is satisfied by the curvature $R^\perp$ of the connection $\nabla^\perp$ of the normal bundle on $M_{1,R}$ (see Lemma \ref{lem-norm-bund-curv-decay}).

We can define the drift Laplacian with respect to $\nabla$ by
\[\mathcal{L}_0 = \Delta - \frac{r}{2} \nabla_{\nabla_g r}.
\]
A section $V \in C^2(\Sigma ; B) $ is an almost eigensection if it satisfies
\[|(\mathcal{L}_0 + \lambda) V| \le Mr^{-2} (|V| + |\nabla V|)
\]
Our generalization of \cite[Theorem 1.2]{Ber} to almost eigensections of vector bundles over weakly conical ends is the following:
\begin{thm}\label{mainthm4} If $(\Sigma^m, g, r)$ is a weakly conical end, $(B, p, \Sigma, h, \nabla)$ is a vector bundle over $\Sigma$ with a metric connection $\nabla$ satisfying Condition \ref{cond-curvature-decay}, and $V \in C^2(\Sigma; B)$ is a section that satisfies
\[ |(\mathcal{L}_0 + \lambda) V| \le M r^{-2} (|V| + |\nabla V|) \text{ and } \int_{\Sigma} (|\nabla V|^2 + |V|^2) r^{2-4\lambda} e^{-\frac{r^2}{4}} < \infty,
\]
then there are constants $R_0$ and $K_0$, depending on $V$, so that for any $R \ge R_0$
\[
\int_{\{r \ge R\}} \bigg( |V|^2 +r^2|\nabla V|^2 + r^4 \bigg| \nabla_{\partial_r} V - \frac{2\lambda}{r}V\bigg|^2 \bigg) r^{-1-m-4\lambda} \le \frac{K_0}{R^{m+4\lambda}} \int_{r = R} |V|^2
\]
Moreover, $V$ is asymptotically homogeneous of degree $2\lambda$ and $\tr_\infty^{2\lambda} V = a$ for some section $a \in L^2(L(\Sigma); B|_{L(\Sigma)})$ that satisfies $\alpha^2 = \lim_{\rho \rightarrow \infty} \rho^{1-m-4\lambda} \int_{\{r = \rho\}} |V|^2 = \int_{L(\Sigma)} |a|^2$ and, 
\[
\int_{\{ r \ge R\}} \bigg( |V|^2 + r^2 (|V- A|^2 + |\nabla V|^2) + r^4 \bigg| \nabla_{\partial_r} V - \frac{2\lambda}{r}V\bigg|^2 \bigg) r^{-2-m-4\lambda}  \le \frac{K_0 \alpha^2}{R^2}.
\]
Here $A \in L^2_{loc}(\Sigma; B)$ is the leading term of $V$ and $L(\Sigma)$ is the link of the asymptotic cone. 
\end{thm}

The proof of Theorem \ref{mainthm4} is almost identical to the proof of Theorem 1.2 in \cite{Ber}. However, it is necessary to prove several preliminary estimates and lemmas in the vector-valued case before Bernstein's arguments can be directly applied. We introduce some of the basic notation and terminology from \cite{Ber}, prove the missing results, and indicate their scalar-valued analogues in the original paper. After substituting these modified results, one can follow Bernstein's proofs directly. Note that if a proposition or proof is not included in the sequel, we implicitly assume that it can be extended to the vector-valued case essentially without modification.

\subsection{Basic Definitions and Estimates}
We now introduce some notation and facts about weakly conical ends from \cite{Ber}. The first condition in \eqref{radderiv} ensures that $r$ has no critical points for sufficiently large $\rho \in (R_\Sigma, \infty)$ and $r$ is proper. Thus the level sets $S_\rho = r^{-1}(\rho)$ are compact $C^2$-regular hypersurfaces foliating $\Sigma$. Let $L(\Sigma)= S_{R_L} = S_{R_\Sigma + 1}$ be the \textit{link} of $\Sigma$. Notice that the bound on $\nabla_g r$ guarantees that $L(\Sigma)$ has the same topological type as any other $S_\rho$. 

For any $R \ge R_\Sigma$, let 
\[
E_R = \{p \; : \; r(p) >R\}
\]
and for $R_2 > R_1 \ge R_\Sigma$ define the annuli
\[
A_{R_2,R_1} = E_{R_1}\setminus \bar{E}_{R_2}.
\]
Define three important $C^1$ vector fields on $\Sigma$:
\[
\partial_r = \nabla_g r \;, \; \mathbf{N} = \frac{\nabla_g r}{|\nabla_g r|}\;, \; \text{ and } \mathbf{X} = r \frac{\nabla_g r}{|\nabla_g r|^2}.
\]
We state the following identities without proof. Derivations can be found in Section 2 of \cite{Ber}. The divergences of these three vector fields are 
\begin{equation}\label{divvec}
\divg_g \partial_r = \divg_g \mathbf{N} = \frac{m-1}{r} + O(r^{-3}) \; \text{ and } \divg_g \mathbf{X} = m + O(r^{-2}).
\end{equation}
The asymptotics of the Hessian are
\begin{equation}\label{asymphess}
\nabla^2_g r = \frac{g - dr \otimes dr}{r} + O(r^{-3}).
\end{equation}
For any $C^1$ vector fields $\mathbf{Y}, \mathbf{Z} \in \Gamma(T\Sigma)$
\begin{equation}\label{distsqrdhess}
\mathbf{Y}|\nabla_g r|^2 = O(|\mathbf{Y}|r^{-3})
\end{equation}
\begin{equation}\label{covderivest}
g(\nabla_{\mathbf{Y}}\mathbf{X}, \mathbf{Z}) = g(\mathbf{Y}, \mathbf{Z}) + O(|\mathbf{Y}||\mathbf{Z}|r^{-2})
\end{equation}
If $A_{S_\rho}$ and $H_{S_\rho}$ are the second fundamental form and mean curvature of $S_\rho$, then
\begin{equation}\label{asympcurv}
A_{S_\rho} = \rho^{-1}g_{S_\rho} + O(\rho^{-3}), \;\; H_{S_\rho} = \frac{m-1}{\rho} + O(\rho^{-3}).
\end{equation}
For any $\tau \ge 1$ let 
\[\Pi_\tau : \Sigma \rightarrow E_{\tau R_\Sigma} \]
be the time $\ln \tau$ flow of $\X$ on $\Sigma$. As $\X \cdot r = r$, $\Pi_\tau (S_\rho) = S_{\tau \rho}$ and $\Pi_\tau (E_\rho) = E_{\tau \rho}$. The restriction $\Pi_\tau$ to the link $L(\Sigma)$ is a diffeomorphism $\pi_\tau : L(\Sigma) \rightarrow S_{\tau R_L}$. Let $g(\tau)$ be the $C^1$-Riemannian metric induced by $g$ on $S_\tau$. Take the limit $C^0$-metric on $L(\Sigma)$,
\[
g_L = \lim_{\tau \rightarrow \infty} \pi_\tau^* (r^{-2} g(\tau)).
\]
This limit metric on the link extends to a $C^0$ cone metric on $\Sigma$. Let $g_\tau = \tau^{-2}\Pi_\tau^* g$:
\[
g_C = \lim_{\tau \rightarrow \infty} g_\tau = dr^2 + r^2 g_L .
\]
The $C^0$-Riemannian cone $(\Sigma, g_C)$ with link $(L(\Sigma), g_L)$ is the asymptotic cone of $(\Sigma, g, r)$. Let $d\mu_C$, $d\mu_{g_\tau}$, $d\mu_{g}$ be the densities associated to the metrics $g_C$, $g_\tau$, $g$ respectively. For any compact set $K \subset \Sigma$ there is a $\tau_0 = \tau_0(K)$ so that for any $W \in C^1(K;B)$ and any $\tau \ge \tau_0$,
\begin{equation}\label{L2conetog}
\frac{1}{2} \int_K |W|^2 d\mu_C \le \int_K |W|^2 d\mu_{g_\tau} \le \int_K |W|^2 \Pi^*_\tau (r^{-m} d\mu_g) \le 2 \int_K |W|^2 d\mu_C.
\end{equation}

We define the pullback of a section of a vector bundle as follows. Let $G \in L^2_{loc}(\Sigma; B ; d\mu_g)$ be a section, and also denote by $G$ a choice of pointwise a.e. representative of the equivalence class $[G]$. Then, we define the pullback of $G$ at the point $p$ by the flow $\Pi_\tau$ as follows:
\[
\Pi_\tau^*G (p) = P_{\Pi_t (p), \tau}(G(\Pi_\tau (p))) \in B_p,
\]
where $P_{c,t}$ is defined as the parallel transport from the fiber of $B$ above $c(t)$ to the fiber above $c(0)$ along the path $c: (a,b) \rightarrow \Sigma$. Since $\Pi_t (p)$ is a regular path for fixed $p$, this operation is well defined. 

A section $F \in L^2_{loc}(\Sigma; B ; d\mu_C)$ is homogeneous of degree $d$ if, for all $\tau \ge 1$, $r^{-d}F$ is preserved by the pullback--that is, $\Pi_\tau^* (r^{-d} F) = r^{-d} F$. If $F$ is homogeneous of degree $d$, the restriction $r^{-d}F$ to $L(\Sigma)$ is a well-defined $L^2$ section of the bundle $B$ restricted to the link $L(\Sigma)$. We call this restriction the degree $d$ trace of $F$.
\[f = \tr^d(F)\]
A section $G \in L^2_{loc}(\Sigma; B ; d\mu_g)$ is called asymptotically homogeneous of degree $d$ if 
\[
\lim_{\tau \rightarrow \infty} \Pi_\tau^* (r^{-d} G) = r^{-d}F \text{ in }L^2_{loc}(\Sigma; B; d\mu_C)
\]
where $F$ is a homogeneous section of degree $d$. The section $F$ on the limit cone is called the leading term of $G$. The trace at infinity of $G$ is defined to be the degree $d$ trace of the leading term $F$.
\[\tr^d_\infty(G) = \tr^d(F). \]

\subsection{Extension to the Case of Almost \(\mathcal{L}_{\mu}\)-harmonic Sections}
Let $(\Sigma^m, g, r)$ be a weakly conical end of dimension $m$, and let $(B, p, \Sigma, h, \nabla)$ be a vector bundle of rank $n$ over $\Sigma$ with projection map $p$, bundle metric $h$, and metric connection $\nabla$ satisfying Condition \ref{cond-curvature-decay}. Define the weight function 
\[\Phi_{\m} : \R^+ \rightarrow \R^+, \;\;\Phi_{\m}(t) = t^\m e^{-\frac{t^2}{4}}
\]
As in \cite{Ber}, we will write $\Phi_\m(p)$ instead of $\Phi_\m(r(p))$. The associated drift Laplacian on sections $V \in \Gamma(B)$ will be
\[ \mathcal{L}_\m V = \Delta V - \frac{r}{2} \nabla_{\partial_r} V + \frac{\m}{r}\nabla_{\partial_r} V, 
\]
where $\Delta$ is the trace with respect to $g$ of the double covariant derivative $\nabla^2$ on the vector bundle $B$. Almost $\mathcal{L}_\mu$-harmonic sections are sections that satisfy
\begin{equation}\label{almostharmonic}
|\mathcal{L}_\mu V|_h \le M r^{-2} (|V|_h + |\nabla V|_h),
\end{equation}
where $V$ and $\nabla V$ are contracted with respect to the tensors $h_{\alpha \beta}$ and $h_{\alpha \beta} g^{ij}$, respectively. As is standard, coordinates on the base $\Sigma$ will be given by Latin indices, and coordinates in the fibers $B_x$ will be given by Greek indices. In the rest of the paper, we suppress the mention of specific metric tensors unless ambiguity arises.

As a preliminary, we must define appropriate function spaces for our analysis. For any $R> R_\Sigma$, let $C^l(\bar{E}_R; B)$ be the space of $l$-times differentiable sections of $B$ defined on $E_R$. Consider the weighted $L^2$ norms on sections in $C^l(\bar{E}_R; B)$: 
\[||V||^2_\m = ||V||^2_{\m,0} = \int_{\bar{E}_R} |V|^2 \Phi_\m \] 
and
\[
||V||^2_{\m,1} =||V||^2_\m  + ||\nabla V||^2_\m =  \int_{\bar{E}_R} (|V|^2 + |\nabla V|^2) \Phi_\m.
\]
Let $C^l_\mu(\bar{E}_R; B)$ and $C^l_{\mu,1}(\bar{E}_R; B)$ be the space of $l$-times differentiable sections such that $||V||_\m < \infty$ and $||V||_{\m,1} < \infty$ respectively. 

Fix $R> R_\Sigma$ and a section $V \in C^2(\bar{E_R}; B)$. For each $\rho > R$, define the boundary $L^2$ norm 
\[ 
B(\rho) = \int_{S_\rho} |V|^2 |\nabla_g r| \text{ and } \hat{B}_\m(\rho) = \Phi_\m(\rho) B(\rho),
\]
and define the flux
\[
F(\rho) = \int_{\partial E_\rho} \langle V, \nabla_\mathbf{-N} V \rangle = - \int_{S_\rho} \frac{\langle V, \nabla_{\partial_r} V \rangle}{|\nabla_g r|} \text{ and }\hat{F}_\m(\rho) = \Phi_\m(\rho)F(\rho).
\]
The following quotient is the associated frequency function.
\[
N(\rho) = \frac{\rho F(\rho)}{B(\rho)} = \frac{\rho \hat{F}_\m(\rho)}{\hat{B}_\m(\rho)}
\] 
If $V \in C^2_{\m,1}(\bar{E}_R; B)$, then define the weighted Dirichlet energy 
\[
\hat{D}_\m(\rho) = \int_{\bar{E}_\rho} |\nabla V|^2 \Phi_\m \text{ and } D_\m(\rho) = \Phi_\m(\rho)^{-1} \hat{D}_\m(\rho).
\]
The corresponding frequency function is 
\[
\hat{N}_\m(\rho) = \frac{\rho \hat{D}_\m(\rho)}{\hat{B}_\m(\rho)} = \frac{\rho D_\m(\rho)}{B(\rho)}.
\]
Set
\[
\hat{L}_\m(\rho) = \int_{E_\rho} \langle V, \mathcal{L}_\m V \rangle \Phi_\m \text{ and }L_\m (\rho) = \Phi_\m(\rho)^{-1} \hat{L}_\m(\rho)
\]
If $\mathcal{L}_\m V \in C^0_\m (\bar{E}_R; B)$, as it is when $V \in C^2_{\m,1}(\bar{E}_R; B)$ is almost $\mathcal{L}_\m$-harmonic, we can apply integration by parts for tensor fields and obtain
\begin{align*}
\hat{L}_\m(\rho)  &= \int_{E_\rho} \langle V, \mathcal{L}_\m V \rangle \Phi_\m  \\
&= \int_{E_\rho} \langle \Phi_\mu V, \Delta V \rangle + \int_{E_\rho} \langle V, -\frac{r}{2} \nabla_{\partial_r} V +\frac{\mu}{r} \nabla_{\partial_r} V \rangle \Phi_\mu
\\
&= - \int_{E_\rho} \langle \nabla (\Phi_\m V), \nabla V \rangle - \int_{S_\rho} \langle V \otimes \mathbf{N}^*, \nabla V \rangle \Phi_\m \\
& \hspace{3cm} + \int_{E_\rho} \langle V, - \frac{r}{2} \nabla_{\partial_r} V + \frac{\m}{r}\nabla_{\partial_r} V \rangle \Phi_\m \\
&= -\int_{E_\rho} \langle d \Phi_\m \otimes V, \nabla V \rangle  -\int_{E_\rho} \langle \nabla V, \nabla V \rangle \Phi_\m \\
& \;\;\;\;\;\;\;\;\;\;\; + \Phi_\m (\rho) \int_{S_\rho} \langle V , \nabla_{-\mathbf{N}}V \rangle + \int_{E_\rho} \langle V, - \frac{r}{2} \nabla_{\partial_r} V + \frac{\m}{r}\nabla_{\partial_r} V \rangle \Phi_\m \\
&= -\int_{E_\rho} \bigg(\frac{\m}{r} - \frac{r}{2} \bigg) \langle V, \nabla_{\partial_r} V \rangle \Phi_\m - \hat{D}_\m(\rho) + \hat{F}_\m(\rho) 
\\
& \hspace{3cm} + \int_{E_\rho} \langle V, - \frac{r}{2} \nabla_{\partial_r} V + \frac{\m}{r}\nabla_{\partial_r} V \rangle \Phi_\m \\
&= - \hat{D}_\m(\rho) + \hat{F}_\m(\rho)
\end{align*} 
This yields the useful identities 
\begin{equation}\label{dirichlet}
\hat{F}_\m(\rho) = \hat{D}_\m(\rho) + \hat{L}_\m(\rho) \text{ and } F(\rho) = D_\m(\rho) + L_\m(\rho). 
\end{equation}

We prove a Poincar\'{e} inequality analogous to \cite[Proposition 3.1]{Ber}. 
\begin{prop}\label{poincareineq}
There is an $R_P = R_P(\Lambda, \m, m)$ so that if $R \ge R_P$ and $V \in C^2_{\m,1}(\bar{E}_R; B)$, then 
\[
\int_{\bar{E}_R} |V|^2 \Phi_\m \le \frac{32}{R^2} \hat{D}_\m(R) + \frac{16}{R} \hat{B}_\m(R)
\]
\end{prop}

\begin{proof}
Define a tangential vector field $\mathbf{W}(p) = r(p)|V(p)|^2 \Phi_\m(p) \mathbf{N}(p)$. Take its divergence:
\[
\divg_g \mathbf{W} = r|V|^2 \Phi_\m \divg_g(\mathbf{N}) + \langle \nabla_g (r|V|^2 \Phi_\m), \mathbf{N} \rangle
\]
By identity \eqref{divvec} the first summand can be rewritten as 
\[
r|V|^2 \Phi_\m \divg_g(\mathbf{N}) = r|V|^2 \Phi_\m \bigg( \frac{m-1}{r} + O(r^{-3}) \bigg) = (m-1)|V|^2 \Phi_\m + |V|^2\Phi_\m O(r^{-2})
\]
Expand the second summand
\begin{align*}
\langle \nabla (r|V|^2 \Phi_\m), \mathbf{N} \rangle &=  \langle \nabla_g r,\mathbf{N}\rangle |V|^2 \Phi_\m + \langle \nabla_g r,\mathbf{N}\rangle r\bigg( \frac{\m}{r}-\frac{r}{2} \bigg)\Phi_\m |V|^2 + r \Phi_\m 2 h(\nabla_{\mathbf{N}} V, V)  \\
&=  |\nabla_g r||V|^2 \Phi_\m + |\nabla_gr|\bigg( \m-\frac{r^2}{2} \bigg)\Phi_\m |V|^2 + r \Phi_\m 2 h(\nabla_{\mathbf{N}} V, V)  \\
&= (1 + O(r^{-4}))\bigg(1 + \m - \frac{r^2}{2}\bigg)|V|^2 \Phi_\m + r \Phi_\m 2h(\nabla_{\mathbf{N}}V, V) \\
&= \bigg(1 + \m - \frac{r^2}{2}\bigg)|V|^2 \Phi_\m + |V|^2 \Phi_\m O(r^{-2};\m) + r \Phi_\m 2h(\nabla_{\mathbf{N}}V, V)\\
\end{align*}
Combining the two summands, we obtain 
\[\divg_g \mathbf{W} = (m-1)|V|^2 \Phi_\m + \bigg(1 + \m - \frac{r^2}{2}\bigg)|V|^2 \Phi_\m + r \Phi_\m 2h(\nabla_{\mathbf{N}}V, V) +  |V|^2 \Phi_\m O(r^{-2};\m)
\]
We use Cauchy-Schwartz for $h$ and the absorbing inequality to estimate
\begin{align*}
r \Phi_\m 2h(\nabla_{\mathbf{N}}V, V) & \le r \Phi_\m 2 |\nabla_{\mathbf{N}} V|_h |V|_h \\
&\le \frac{r^2 |V|^2 \Phi_\m}{4} + 4|\nabla_{\mathbf{N}}V|^2 \Phi_\m \\
&\le \frac{r^2 |V|^2 \Phi_\m}{4} + 4|\nabla V|^2 \Phi_\m,
\end{align*}
Combining terms, we obtain
\[
\divg_g \mathbf{W} \le (m+\m)|V|^2 \Phi_\m - \frac{r^2}{4} |V|^2 \Phi_\m + 4|\nabla V|^2 \Phi_\m + |V|^2 \Phi_\m O(r^{-2};\m)
\]
For $R>R_P$ sufficiently large, the $r^2$ term will dominate, and the following inequality holds
\[
\divg_g \mathbf{W} \le 4|\nabla V|^2 \Phi_\m - \frac{r^2}{8} |V|^2 \Phi_\m.
\]
As $V \in C^2_{\m,1}(\bar{E}_R)$, the proposition follows immediately from the divergence theorem.
\end{proof}

Next, we derive the formula for the derivative of the boundary $L^2$ norm $B(\rho)$, which should be compared to \cite[Lemma 3.2]{Ber}. 

\begin{lem}\label{dboundarynorm}
We have 
\[B'(\rho) = \frac{m-1}{\rho} B(\rho) - 2 F(\rho) + B(\rho) O(\rho^{-3})
\]
and 
\[
\hat{B}'_\m(\rho) = \frac{m+\m-1}{\rho} \hat{B}_\m(\rho) - \frac{\rho}{2} \hat{B}_\m(\rho) - 2 \hat{F}_\m (\rho) + \hat{B}_\m(\rho)O(\rho^{-3})
\]
\end{lem}

\begin{proof}
Calculate the first variation with respect to the vector field $r^{-1}\mathbf{X}$. Let $\nu(s) = (\det g(x+ sr^{-1}\mathbf{X}))^{1/2}(\det g(x))^{-1/2}$. 
\begin{align*}
\frac{d}{ds} B(\rho + s) \bigg|_{s=0} &= \frac{d}{ds}\bigg|_{s=0} \int_{S_\rho} |V(x + sr^{-1}\mathbf{X})|^2|\nabla_g r|(x + sr^{-1}\mathbf{X})\nu(s) dVol_{S_\rho}\\
&= \int_{S_\rho} 2\langle V, \nabla_{r^{-1}\mathbf{X}} V\rangle |\nabla_g r| dVol_{S_\rho} + \int_{S_\rho} |V|^2 g(\nabla |\nabla_g r|, r^{-1}\mathbf{X} ) dVol_{S_\rho} \\ 
&\;\;\;\;\;\;\;\;\;\;\;\;\;\;\;+ \int_{S_\rho} |V|^2 |\nabla_g r| g(r^{-1}\mathbf{X},\mathbf{N}) H_{S_\rho} dVol_{S_\rho}\\
&= \int_{S_\rho} 2\langle V, \nabla_{\mathbf{N}} V\rangle dVol_{S_\rho} + \int_{S_\rho}\frac{|V|^2}{2|\nabla_g r|} g(\nabla |\nabla_g r|^2, r^{-1}\mathbf{X} ) dVol_{S_\rho} \\
&\;\;\;\;\;\;\;\;\;\;\;\;\;\;\;+ \int_{S_\rho} |V|^2 |\nabla_g r| g(r^{-1}\mathbf{X},\mathbf{N}) H_{S_\rho} dVol_{S_\rho}\\
&= - 2F(\rho) + \frac{m-1}{\rho}B(\rho) + O(\rho^{-3})B(\rho) \\
& \hspace{3cm} + \frac{1}{2}\int_{S_\rho}|V|^2 g(\nabla |\nabla_g r|^2, r^{-1}\mathbf{X} )|\nabla_g r|^{-1} dVol_{S_\rho} \\
&= - 2F(\rho) + \frac{m-1}{\rho}B(\rho) + O(\rho^{-3})B(\rho)
\end{align*}
We use \eqref{asympcurv} in the fourth equality and \eqref{distsqrdhess} in the final equality. This completes the proof. The second identity follows immediately.
\end{proof}

Now we compute the rate of change of the weighted Dirichlet energy $\hat{D}_\m(\rho)$, analogous to \cite[Proposition 4.2]{Ber}. The need to commute second covariant derivatives creates an error term that does not appear in \cite{Ber}, but appears in the adaptation to Ricci expanders by Deruelle and Schulze \cite[Proposition 5.1]{DerSch}. 

\begin{prop}\label{ddirichletenergy}
If $V \in C^2_{\m,1}(\bar{E}_R; B)$, $\mathcal{L}_\m V \in C^0_\m(\bar{E}_R; B)$ and $\rho \ge R$, then 
\begin{align*}
\hat{D}'_\m(\rho) &= - \frac{2}{\rho} \int_{\bar{E}_\rho} \langle \nabla_{\mathbf{X}} V, \mathcal{L}_\m V \rangle \Phi_\m - 2\int_{S_\rho} \frac{ |\nabla_{\mathbf{N}} V|^2}{|\nabla_g r|} \Phi_\m + \bigg( \frac{m+\m-2}{\rho} -\frac{\rho}{2} \bigg) \hat{D}_\m(\rho) \\
&\;\;\;\;\; - \frac{1}{\rho} \int_\rho^\infty t\hat{D}_\m(t) dt + O(\rho^{-3})\hat{D}_\m(\rho) + O(\rho^{-\frac{9}{2}})\hat{D}_\mu^\frac{1}{2}(\rho)\hat{B}^\frac{1}{2}(\rho).
\end{align*}
\end{prop}

\begin{proof}
By the coarea formula,
\[
\hat{D}_\m(\rho) = \int_\rho^\infty \Phi_\m(\sigma) \int_{S_\sigma} \frac{|\nabla V|^2}{|\nabla_g r|} dV_{S_\sigma} d\sigma.
\]
Differentiating once and applying \eqref{divvec} yields
\begin{align*}
\hat{D}'_\m(\rho) &= - \Phi_\m(\rho) \int_{S_\rho} \frac{|\nabla V|^2}{|\nabla_g r|} dV_{S_\rho} = \rho^{-1} \int_{S_\rho} |\nabla V|^2 \Phi_\m(r) g(\mathbf{X}, -\mathbf{N}) \\
&= \frac{1}{\rho} \int_{E_\rho} \divg_g (|\nabla V|^2 \Phi_\m(r) \mathbf{X} )\\
&= \frac{1}{\rho} \int_{E_\rho} \bigg( m + \m - \frac{r^2}{2} + O(r^{-2})\bigg)|\nabla V|^2 \Phi_\m + \nabla_{\mathbf{X}} |\nabla V|^2 \Phi_\m
\end{align*}
We analyze the last summand. By metric compatibility with covariant differentiation,
\begin{align*} 
\nabla_\X \langle \nabla V, \nabla V \rangle = 2 \langle \nabla_\X (\nabla V), \nabla V \rangle = 2\langle \nabla^2V(\X, \cdot), \nabla V \rangle
\end{align*}
We derive a weighted Rellich-Ne\v{c}as identity for vector bundles to apply the divergence theorem and extract an $\mathcal{L}_\m$ term from the Hessian term. We take the divergence of the 1-form $L(Y) = \langle \nabla_\X V, \nabla_Y V \rangle$ on the tangent bundle.
\begin{align*}
\divg_g (\langle \nabla_\X V, \nabla V \rangle \Phi_\m) &= g^{ij}(\nabla_{\partial_j} \langle \nabla_\X V, \nabla_{\partial_i}V \rangle )\Phi_\m + g^{ij} \nabla_{\partial_j}\Phi_\m \langle \nabla_{\X} V, \nabla_{\partial_i} V \rangle  \\
&= g^{ij} (\langle \nabla_{\partial_j} (\nabla_\X V) , \nabla_{\partial_i} V \rangle + \langle \nabla_\X V, \nabla_{\partial_j} (\nabla_{\partial_i} V)\rangle )\Phi_\m\\
& \hspace{3cm} + g^{ij} \nabla_{\partial_j}\Phi_\m \langle \nabla_{\X} V, \nabla_{\partial_i} V \rangle\\
&= g^{ij}(\langle \nabla^2 V(\partial_j, \X), \nabla_{\partial_i} V \rangle + \langle \nabla_{\nabla_{\partial_j}\X} V, \nabla_{\partial_i} V \rangle \\
& \;\;\;\;\;\;\;\;\;\; + \langle \nabla_\X V, \nabla^2 V(\partial_i, \partial_j) \rangle + \langle \nabla_\X V, \nabla_{\Gamma_{ij}^k \partial_k} V \rangle )\Phi_\m \\
& \;\;\;\;\;\;\;\;\;\;  + g^{ij} \nabla_{\partial_j}\Phi_\m \langle \nabla_{\X} V, \nabla_{\partial_i} V \rangle 
\end{align*}
We identify the $\mathcal{L}_\m V$ term:
\[ g^{ij}\langle \nabla_\X V, \nabla^2 V(\partial_i, \partial_j) \rangle \Phi_\m + g^{ij}\langle  \nabla_\X V, \nabla_{\Gamma_{ij}^k \partial_k} V \rangle, )\Phi_\m + g^{ij} \nabla_{\partial_j}\Phi_\m \langle \nabla_{\X} V, \nabla_{\partial_i} V \rangle\]
\begin{align*}
&= \langle \nabla_\X V, \Delta V \rangle \Phi_\m + (g^{ij}\partial_j r) \bigg(\frac{\m}{r} - \frac{r}{2}\bigg)\Phi_\m \langle \nabla_\X V, \nabla_{\partial_i} V \rangle \\
&= \langle \nabla_\X V, \Delta V \rangle \Phi_\m + \bigg(\frac{\m}{r} - \frac{r}{2}\bigg)\Phi_\m \langle \nabla_\X V, \nabla_{\nabla_g r} V \rangle \\
&=\langle \nabla_\X V, \mathcal{L}_\m V \rangle \Phi_\m
\end{align*}
We now have 
\begin{align*}
\divg_g (\langle \nabla_\X V, \nabla V \rangle \Phi_\m) &= \langle \nabla^2V(\cdot, \X ), \nabla V \rangle \Phi_\m + \langle \nabla_\X V, \mathcal{L}_\m V \rangle \Phi_\m \\
& \hspace{2cm} + g^{ij}\langle \nabla_{\nabla_{\partial_j}\X} V, \nabla_{\partial_i} V \rangle \Phi_\m
\end{align*}
To evaluate the final term, we calculate at a point $p$ and assume that $(x^i)_{i=1}^n$ are normal coordinates centered at $p$, and we have a local geodesic frame $\{e_\alpha\}_{\alpha =1}^k$ of the fiber space, such that $\nabla_{\partial_i} e_\beta (p) = 0$, for all $\beta =1, \ldots, n$ and $i=1, \ldots m$. Then, at the point $p$, we find that 
\begin{align*}
g^{ij}\langle \nabla_{\nabla_{\partial_j}\X} V, \nabla_{\partial_i} V \rangle &= \delta^{ij} \langle \partial_j\X^k\partial_k V^\alpha e_\alpha, \partial_i V^\beta e_\beta \rangle \\
&=\partial_i \X^k \delta_{\alpha \beta} \partial_k V^\alpha \partial_i V^\beta \\
&= \sum_{\alpha =1}^k g( (\partial_i V^\alpha \partial_i \mathbf{X}^k)\partial_k, \partial_k V^\alpha \partial_k)\\
&= \sum_{\alpha =1}^k g(\nabla_{\nabla V^\alpha} \X, \nabla V^\alpha) \\
&= \sum_{\alpha = 1}^k \bigg( |\nabla V^\alpha|_g^2 + O(r^{-2})|\nabla V^\alpha|_g^2 \bigg) \\
&= |\nabla V|^2 + O(r^{-2})|\nabla V|^2, 
\end{align*}
where the second to last equality comes from equation \eqref{covderivest}. The Rellich-Ne\v{c}as identity is
\begin{align*}\langle \nabla^2V(\cdot, \mathbf{X}), \nabla V \rangle \Phi_\m &= \divg_g (\langle \nabla_\mathbf{X} V, \nabla V \rangle \Phi_\m) - \langle \nabla_\mathbf{X} V, \mathcal{L}_\m V \rangle \Phi_\m \\
& \hspace{1cm} - |\nabla V|^2\Phi_\m - O(r^{-2})|\nabla V|^2\Phi_\m
\end{align*}
Before we can plus this identity into our original expression, we must commute $\X$ to the first position in the Hessian $\nabla^2 V$. We obtain
\begin{align*}\langle \nabla^2V(\mathbf{X},\cdot), \nabla V \rangle \Phi_\m &= \divg_g (\langle \nabla_\mathbf{X} V, \nabla V \rangle \Phi_\m) - \langle \nabla_\mathbf{X} V, \mathcal{L}_\m V \rangle \Phi_\m  \\
& \qquad - \langle R^\nabla(\X, \cdot )V, \nabla V \rangle \Phi_\m - |\nabla V|^2\Phi_\m + O(r^{-2})|\nabla V|^2\Phi_\m.
\end{align*}
We assume that $R^\nabla$ satisfies Condition \ref{cond-curvature-decay} and obtain
\begin{align*}\langle \nabla^2V(\mathbf{X},\cdot), \nabla V \rangle \Phi_\m &= \divg_g (\langle \nabla_\mathbf{X} V, \nabla V \rangle \Phi_\m) - \langle \nabla_\mathbf{X} V, \mathcal{L}_\m V \rangle \Phi_\m  \\
& \qquad - |\nabla V|^2\Phi_\m + O(r^{-2})|\nabla V|^2\Phi_\m + O(r^{-3})|V||\nabla V|\Phi_\m.
\end{align*}

Plug this into our original expression and apply the divergence theorem. 
\begin{align}\label{eq-derivative-dirichlet}
\hat{D}'_\m(\rho) &= \frac{1}{\rho}\int_{E_\rho} \bigg(-2\langle \nabla_\X V, \mathcal{L}_\m V\rangle + ( m + \m - 2) |\nabla V|^2 - \frac{r^2}{2}|\nabla V|^2 \bigg) \Phi_\m \nonumber \\
&\qquad +O(\rho^{-4})\int_{E_\rho} |V||\nabla V|\Phi_m - \frac{2}{\rho} \int_{S_\rho}  \langle \nabla_{\mathbf{X}} V, \nabla_{\mathbf{N}} V\rangle \Phi_\m + O(\rho^{-3}) \hat{D}_\m(\rho).
\end{align}
Estimate the second integral in \eqref{eq-derivative-dirichlet} using the Cauchy-Schwarz inequality and Proposition \ref{poincareineq}. 
\begin{align*}\label{eq-derivative-dirichlet}
\int_{E_\rho} |V||\nabla V|\Phi_\m &\le \bigg(\int_{E_\rho} |\nabla V|^2 \Phi_\m\bigg)^\frac{1}{2}\bigg(\int_{E_\rho} |V|^2 \Phi_\m\bigg)^\frac{1}{2} \\
& \le \hat{D}_\m^\frac{1}{2}(\rho)\bigg( \frac{8}{\rho}\hat{D}^\frac{1}{2}_\mu(\rho) + \frac{4}{\sqrt{\rho}}\hat{B}^\frac{1}{2}_\m (\rho) \bigg)\\
& \le O(\rho^{-1})\hat{D}_\m(\rho) + O(\rho^{-\frac{1}{2}})\hat{D}_\m^\frac{1}{2}(\rho) \hat{B}^\frac{1}{2}_\m (\rho)
\end{align*}
We plug this bound into \eqref{eq-derivative-dirichlet} and obtain
\begin{align*}
\hat{D}'_\m(\rho) &= \frac{1}{\rho}\int_{E_\rho} \bigg(-2\langle \nabla_\X V, \mathcal{L}_\m V\rangle + ( m + \m - 2) |\nabla V|^2 - \frac{r^2}{2}|\nabla V|^2 \bigg) \Phi_\m \\
&\qquad - \frac{2}{\rho} \int_{S_\rho}  \langle \nabla_{\mathbf{X}} V, \nabla_{\mathbf{N}} V\rangle \Phi_\m + O(\rho^{-3}) \hat{D}_\m(\rho) +  O(\rho^{-\frac{9}{2}})\hat{D}_\m^\frac{1}{2}(\rho) \hat{B}^\frac{1}{2}_\m (\rho)\\
&= \frac{1}{\rho} \int_{E_\rho} \bigg( -2\langle \nabla_\X V, \mathcal{L}_\m V\rangle + \bigg(\frac{\rho^2}{2} - \frac{r^2}{2} \bigg) |\nabla V|^2 \Phi_\m \\
&\qquad + \bigg( \frac{m+\m-2}{\rho} - \frac{\rho}{2} \bigg)\hat{D}_\m(\rho) - 2 \int_{S_\rho} \frac{|\nabla_{\mathbf{N}} V|^2}{|\nabla_g r|} \Phi_\m \\
&\qquad \qquad + O(\rho^{-3}) \hat{D}_\m(\rho) +  O(\rho^{-\frac{9}{2}})\hat{D}_\m^\frac{1}{2}(\rho) \hat{B}^\frac{1}{2}_\m (\rho).
\end{align*}

Exactly as in \cite{Ber}, one establishes using the coarea formula and Fubini's theorem that 
\[\int_\rho^\infty t \hat{D}_\m(t) dt = \int_{E_\rho} \bigg( \frac{r^2}{2} - \frac{\rho^2}{2} \bigg) |\nabla V|^2 \Phi_\m
\]
This completes the proof of the proposition. 
\end{proof}

By replacing \cite[Proposition 4.2]{Ber} and \cite[Lemma 3.2]{Ber} by Lemma \ref{dboundarynorm} and Proposition \ref{ddirichletenergy} above, one can follow the proof of \cite[Corollary 4.3]{Ber} almost exactly to obtain the following corollary.

\begin{cor}\label{cor-dfreq}
If $V \in C^2_{\m,1}(\bar{E}_R; B)$, $\mathcal{L}_\m V \in C^0_\m(\bar{E}_R; B)$ and $B(\rho) >0 $ at $\rho \ge R$, then, 
\begin{align*} \hat{N}'_\m (\rho) &= \frac{-2 \int_{E_\rho} \langle \nabla_{\mathbf{X}} V + N(\rho)V, \mathcal{L}_\m V \rangle \Phi_\m - \int_{\rho}^\infty t\hat{D}_\m(t) dt }{\hat{B}_\m(\rho)} \\
&- \frac{\frac{2}{\rho} \int_{S_\rho} |\nabla_{\mathbf{X}} V + N(\rho) V|^2 |\nabla_g r| \Phi_\m}{\hat{B}_\m(\rho)} + \hat{N}_\m(\rho) O(\rho^{-3}) + \hat{N}^\frac{1}{2}_\m(\rho) O(\rho^{-4}) 
\end{align*}
\end{cor}

\begin{proof}
The proof proceeds exactly as in \cite{Ber}, except that we obtain the last term by noticing that
\[
\frac{\rho \hat B_\m (\rho) \hat D_\m^\frac{1}{2}(\rho)\hat B_\m^\frac{1}{2}(\rho)}{\hat B_\mu^2 (\rho)}O\big(\rho^{-\frac{9}{2}}\big) = \hat N_\m^\frac{1}{2}(\rho)O(\rho^{-4}).
\]
\end{proof}

The modified statements in Proposition \ref{ddirichletenergy} and Corollary \ref{cor-dfreq} are used in four places in the arguments on pp. 9-15 \cite{Ber} to prove the central frequency decay estimates: the proofs of \cite[Proposition 4.4]{Ber}, \cite[Proposition 4.5]{Ber}, \cite[Lemma 5.3]{Ber}, and \cite[Proposition 5.4]{Ber}. The extra terms affect the proofs only very slightly, but we will recall some of these results and pinpoint exactly how to absorb these terms.

\begin{prop}\cite[Proposition 4.4]{Ber} If $V \in C^2_{\m+2,1}(\bar E_R ; B)$ is non-trivial and satisfies \eqref{almostharmonic}, then there is an $R$ depending on $n,m, M, \Lambda, \mu$ so that if $\rho \ge R$, then
\[ N'(\rho) \le \frac{1}{4}\rho^{-1}|N(\rho)| + O(1; M, \mu).\]
\end{prop}

\begin{proof}
The AM-GM inequality allows us to absorb the term $O(\rho^{-\frac{9}{2}})\hat{D}_\mu^\frac{1}{2}(\rho)\hat{B}^\frac{1}{2}(\rho)$ into the term $(\hat{D}_\mu(\rho) + \rho^{-1}\hat{B}_\m(\rho))O(\rho^{-2}; M)$ in the first upper bound for $\hat D'_\m (\rho)$.  
\end{proof}

The extra terms are absorbed trivially into the inequality given by \cite[Corollary 4.5]{Ber}, and the analogous result in our setting holds. Moving into \S 5 of \cite{Ber}, we note the following decay condition on the frequency functions for $V \in C^2_{\m+2,1}(\bar E_R ; B)$: For all $\rho \ge R$
\begin{equation}\label{eq-freq-gamma-decay}
\hat N_\mu (\rho) \le \eta \rho^{2\gamma} \le \eta
\end{equation}
where $\gamma \in [-1,0]$ and $\eta > 0$. 

\begin{lem}\cite[Lemma 5.3]{Ber} If $V \in C^2_{\m+2,1}(\bar E_R ; B)$ is non-trivial and satisfies \eqref{almostharmonic} and \eqref{eq-freq-gamma-decay} for some $\gamma$ and $\eta$, then there are constants $\check{R} \ge R$ and $\check{K}$ depending on $m,n,\Lambda,M, \mu, \eta$, so that, for $\rho \ge \check{R}$
\[
\int_{E_\rho} \langle \mathbf{N}, V \rangle^2 \Phi_\m \le \frac{\check{K}}{\rho} \bigg( \hat N_\m(\rho) - \hat N_\m(\rho +2)\bigg) \hat B_\m (\rho) + \check{K}\rho^{-2+2\gamma}\hat B_\m (\rho).
\]
\end{lem}

\begin{proof}
Define $\tilde{D}(\rho) = \rho^{-n} D_m(\rho)$. In the second upper bound for $\tilde D'(\rho)$, we may absorb the term $O(\rho^{-\frac{9}{2}})\hat{D}_\mu^\frac{1}{2}(\rho)\hat{B}^\frac{1}{2}(\rho)$ into the term 
\[
\frac{2K_2}{\rho^{2+n}\Phi_\m (\rho)} \bigg(\frac{1}{\rho^\frac{1}{2}} \hat D_\m^\frac{1}{2} (\rho) \hat B_\m^\frac{1}{2} (\rho) \bigg)
\]
at the cost of possibly increasing a constant. 
\end{proof}

The final use of Corollary \ref{cor-dfreq} is in the proof of \cite[Proposition 5.4]{Ber}.

\begin{prop}\cite[Proposition 5.4]{Ber} If $V \in C^2_{\m+2,1}(\bar E_R ; B)$ is non-trivial and satisfies \eqref{almostharmonic} and \eqref{eq-freq-gamma-decay} for some $\gamma$ and $\eta$, then there are constants $\check{R} \ge R$ and $\Gamma \ge 0$ depending on $V$, so that, for $\rho \ge \check R$, either
\begin{enumerate}
    \item $\hat N_\m (\rho + 2) - \hat N_\m(\rho) \le -4 \rho^{-1} \hat N_\m (\rho)$, or
    \item $\hat N_\m (\rho + 1) - \hat N_\m(\rho) \le -2 \rho^{-1} \hat N_\m (\rho) + \Gamma\rho^{-2+2\gamma}$.
\end{enumerate}
\end{prop}

\begin{proof}
When Corollary \ref{cor-dfreq} (resp. \cite[Corollary 4.3]{Ber}) is applied to bound $\hat N_m'(\rho)$ near the end of the proof, notice that by \eqref{eq-freq-gamma-decay}
\[
 \hat{N}^\frac{1}{2}_\m(\rho) O(\rho^{-4}) = O(\rho^{-4 + \gamma}) = O(\rho^{-2 + 2\gamma}), 
\]
since $\gamma \in [-1,0]$. Thus, it may be absorbed into the term $\tfrac{2}{10}\Gamma\rho^{-2 + 2\gamma}$. The proof for \cite[Proposition 5.4]{Ber} can be continued without further modifications to prove the statement of the Proposition.
\end{proof}

With these results in hand, one can directly apply the arguments on pp. 14-15 \cite{Ber} to prove the crucial frequency decay estimates in \cite[Theorem 4.1]{Ber} and \cite[Theorem 5.1]{Ber}, which we summarize here:

\begin{thm} If $V \in C^2_{\m+2,1}(\bar{E}_R; B)$ satisfies (4.9) and is non-trivial, then
\[\lim_{\rho \rightarrow \infty} \hat{N}_\m (\rho) = \lim_{\rho \rightarrow \infty} N(\rho) = 0.
\]
Furthermore,
\[\lim_{\rho \rightarrow \infty} \rho^2 \hat{N}_\m(\rho) = \xi [V] \in [0,\infty).
\]
In particular, there is a $\rho_{-1} \ge \R$ so that for $\rho \ge \rho_{-1}$ and $\bar{\xi} = \max\{2 \xi[V], 1\}$, 
$\hat{N}_\m(\rho) \le \rho^{-2} \bar{\xi} \le 1$ and $(\xi[V] - K_2) \rho^{-2} \le N(\rho) \le (\xi[V] + K_2)\rho^{-1}$.
\end{thm}

In order to prove the analogue of \cite[Theorem 6.1]{Ber} for vector bundles, it remains to prove a vector-valued version of \cite[Proposition A.1]{Ber}. 

\begin{prop}\label{asymphomog}
If $G \in C^1(\Sigma; B)$ satisfies 
\[ \int_{E_R} r^{-m} |\nabla G|^2 + r^{2-m} |\nabla_{\partial_r} G|^2 d\mu_g \le \alpha^2 R^{-2}
\]
for all $R \ge R_H$, then $G$ is asymptotically homogeneous of degree 0. Moreover, if $F$ is the leading term of $G$, then $F \in L^2_{loc} (\Sigma; B)$ and 
\[
\int_{E_R} r^{-m} |F-G|^2 d\mu_g \le 16 \alpha^2 R^{-2}.
\]
\end{prop}

\begin{proof}
Let $G_\tau = \Pi_\tau^* G$, and take the covariant derivative at $p$ with respect to the velocity of the flow line $\Pi_\tau(p)$, parametrized by $\tau$. Since the flow $\Pi_\tau$ is the $\log \tau$ flow of the vector field $\mathbf{X}$, we have the following equality.
\begin{align*}
\frac{d}{d\tau} G_\tau (p) &= \frac{d}{dh}\bigg|_{h=0} G_{\tau+h} (p) \\
&= \lim_{h \rightarrow 0} \frac{P_{\Pi_{1 + t/\tau}(\Pi_\tau(p)), h} (G_\tau(\Pi_{1+\frac{h}{\tau}}(p))) - G_\tau(p)}{h} \\
&= \nabla_{\frac{1}{\tau}\mathbf{X}} G_\tau(p).
\end{align*}
We show that covariant differentiation by $\frac{1}{\tau} \mathbf{X}$ commutes with pullback via the flow diffeomorphism $\Pi_\tau$.
\begin{align*}\nabla_{\frac{1}{\tau}\mathbf{X}} G_\tau (p) &= \lim_{h \rightarrow 0} \frac{P_{\Pi_{1 + t/\tau}(\Pi_\tau(p)), h} (G_\tau(\Pi_{1+\frac{h}{\tau}}(p))) - G_\tau(p)}{h} \\
&= \lim_{h \rightarrow 0} \frac{P_{\Pi_{1 + t/\tau}(\Pi_\tau(p)), h} (P_{\Pi_t(p),\tau}(G(\Pi_\tau(\Pi_{1+\frac{h}{\tau}}(p))))) - P_{\Pi_t(p),\tau}(G(\Pi_\tau(p)))}{h}\\
&= \lim_{h \rightarrow 0} \frac{P_{\Pi_t(p),\tau}(P_{\Pi_{1 + t/\tau}(\Pi_\tau(p)), h}(G(\Pi_{1+\frac{h}{\tau}}(\Pi_\tau(p)))))) - P_{\Pi_t(p),\tau}(G(\Pi_\tau(p)))}{h} \\
&= P_{\Pi_t (p),\tau} \bigg( \lim_{h\rightarrow 0}\frac{P_{\Pi_{1 + t/\tau}(\Pi_\tau(p)), h}(G(\Pi_{1+\frac{h}{\tau}}(\Pi_\tau(p))))) - G(\Pi_\tau(p))}{h}\bigg) \\
&= P_{\Pi_t (p),\tau} (\nabla_{\frac{1}{\tau}\mathbf{X}} G (\Pi_\tau (p))) = \Pi_\tau^*(\nabla_{\frac{1}{\tau}\mathbf{X}} G(p)) \\
\end{align*}
Let $K = \bar{A}_{\rho_2, \rho_1} \subset \Sigma$ and $R' \ge R \ge 1$. The following application of Cauchy-Schwartz and Fubini's theorem establish that the sequence of pullbacks $G_R$ is Cauchy in $L^2_{loc}(\Sigma; B, \mu_C)$.
\[ \int_K |G_{R'}- G_R|^2 d\mu_C = \int_K \bigg| \int_{R}^{R'} \frac{d}{ds} G_s ds \bigg|^2 d\mu_C = \int_K \bigg| \int_{R}^{R'} \nabla_{\frac{1}{s}\mathbf{X}} G_s ds \bigg|^2 d\mu_C 
\]
\[ = \int_K \bigg| \int_{R}^{R'} \frac{1}{s}\Pi^*_s(\nabla_{\mathbf{X}}G) ds \bigg|^2 d\mu_C \le \bigg(\int_{R}^{R'} \frac{1}{s^2}ds\bigg) \bigg(\int^{R'}_R \int_K |\Pi^*_s(\nabla_{\mathbf{X}}G)|^2 d\mu_C ds\bigg) \] 
\[\le \frac{1}{R} \int_R^{R'} \int_K \Pi_s^* |\nabla_{\mathbf{X}}G|^2 d\mu_C ds\]
Note that in the last inequality, we used the fact that $P_{\Pi_t(p), \tau}$ is an isometry from $B_{\Pi_\tau(p)}$ to $B_p$. By (4.8), there exists $\tau_0(K)$ such that for $s \ge R \ge \tau_0(K)$, $d\mu_C \le 2 \Pi_s^*(r^{-m} d\mu_g)$. Thus,
\[
\int_K |G_{R'}- G_R|^2 d\mu_C \le \frac{2}{R} \int_R^{R'} \int_K \Pi_s^* |\nabla_{\mathbf{X}}G|^2 \Pi_s^*(r^{-m} d\mu_g) ds
\]
\[= \frac{2}{R} \int_R^{R'} \int_{\Pi_s(K)}  |\nabla_{\mathbf{X}}G|^2r^{-m} d\mu_g ds.
\]
Using the bounds on $|\nabla_g r|$ and the hypotheses of the proposition, for large $R \ge R_H$,
\[\int_K |G_{R'}- G_R|^2 d\mu_C   \le \frac{8}{R} \int_R^{R'} \int_{\bar{E}_{\rho_1 s}} r^{2-m} |\nabla_{\partial_r} G|^2 d\mu_g ds\]
\[
\le \frac{8\alpha^2}{\rho_1^2 R} \int_R^{R'} \frac{1}{s^2} ds \le \frac{8\alpha^2}{\rho_1^2 R^2}  
\]
Thus, $G_R$ is Cauchy in $L^2(K; B; d\mu_C)$, and has a unique limit $F_K \in L^2(K; B; d\mu_C)$. Since we can find this limit for the annuli $A_{\rho_2, \rho_1}$ which form a compact exhaustion of $\Sigma$, this shows that there is a limit $F \in L^2_{loc}(\Sigma; B; d\mu_C)$ such that
\[\lim_{R \rightarrow \infty} G_R = F \text{ in } L^2_{loc}(\Sigma; B; d\mu_C) 
\]
We show that $F$ is homogeneous of degree zero. For any $\tau \ge 1$,
\[F_\tau = \Pi^*_\tau F = \Pi^*_\tau \lim_{R \rightarrow \infty} G_R= \lim_{R \rightarrow \infty} G_{\tau R} = F.
\]
Since $\Pi^*_\tau G \xrightarrow{L^2_{loc}} F$ as $\tau \rightarrow \infty$, $G$ is asymptotically homogeneous of degree 0.
\end{proof}

Now that we have established the preceding results, the proof of \cite[Theorem 6.1]{Ber} with the appropriate modifications suffices to prove the following theorem. 

\begin{thm}\label{mainthmharmver} If $V \in C^2_{\m+2,1}(\Sigma; B)$ is a section that satisfies \eqref{almostharmonic} then there are constants $R_0$ and $K_0$, depending on $V$, so that for any $R \ge R_0$
\[
\int_{\bar{E}_R} \bigg( |V|^2 +r^2|\nabla V|^2 + r^4 | \nabla_{\partial_r} V|^2 \bigg) r^{-1-m} \le \frac{K_0}{R^{m}} \int_{S_R} |V|^2
\]
Moreover, $V$ is asymptotically homogeneous of degree $0$ and $\tr_\infty^{0} V = a$ for some section $a \in L^2(L(\Sigma); B|_{L(\Sigma)})$ that satisfies $\alpha^2 = \lim_{\rho \rightarrow \infty} \rho^{1-m} \int_{S_\rho} |V|^2 = \int_{L(\Sigma)} |a|^2$ and, 
\[
\int_{\bar{E}_R} \bigg( |V|^2 + r^2 (|V- A|^2 + |\nabla V|^2) + r^4 | \nabla_{\partial_r} V |^2 \bigg) r^{-2-m}  \le \frac{K_0 \alpha^2}{R^2}.
\]
Here $A \in L^2_{loc}(\Sigma)$ is the leading term of $V$ and $L(\Sigma)$ is the link of the asymptotic cone. 
\end{thm}

In order to finish the proof of Theorem \ref{mainthm4}, we prove that almost eigensections of $\mathcal{L}_\m$ corresponding to an eigenvalue $\lambda$ can be transformed into almost eigensections corresponding to a different eigenvalue $\lambda + \nu$ of another operator in the family, $\mathcal{L}_{\m-4\nu}$. In particular, almost eigensections of $\mathcal{L}_0$ with negative eigenvalues $-\lambda < 0$ can be transformed into almost harmonic sections of the operator $\mathcal{L}_{4\lambda}$. 

Before stating the proposition, we note that when we apply the operator $\mathcal{L}_\m$ to sections of different bundles, we will assume that we are taking all covariant derivatives with respect to the particular section to which we are applying it. In particular, when we apply $\mathcal{L}_\m$ to scalar functions, we will assume that we are applying it to sections of the trivial bundle $\Sigma \times \R$ endowed with the standard product metric. 

\begin{prop}\label{harm2eigen}
There is a constant $M' = M'(M, \m, \nu, n, \Lambda)$ such that if $(\Sigma^m, g, r)$ is an asymptotically conical end with associated constant $\Lambda$, $(B,p,\Sigma, h, \nabla)$ is a vector bundle of rank $n$ over $\Sigma$ with metric $h$, and $V \in C^2(\Sigma; B)$ satisfies
\[
|(\mathcal{L}_\m + \lambda) V| \le M r^{-2} (|V| + |\nabla V|) 
\]
then, $\hat{V} = r^{2\nu}V$ satisfies
\[
|(\mathcal{L}_{\m-4\nu} + \lambda + \nu) \hat{V}| \le M'r^{-2}(\hat{V} + |\nabla \hat{V}|).
\]
\end{prop}

\begin{proof}
We calculate directly, for arbitrary $\m' \in \Z$ 
\begin{align*}
\mathcal{L}_{\m'} \hat{V} &= \mathcal{L}_{\m'}(r^{2\nu} V) \\
&= \Delta (r^{2\nu}V) + \bigg( \frac{\m'}{r} - \frac{r}{2} \bigg) \nabla_{\partial_r}(r^{2\nu} V) \\
&= \divg( \nabla(r^{2\nu} V)) + \bigg( \frac{\m'}{r} - \frac{r}{2} \bigg) ((\partial_r r^{2\nu}) V) + r^{2\nu} \nabla_{\partial_r} V) \\
&= \divg( 2\nu r^{2\nu-1} dr \otimes V + r^{2\nu} \nabla V) + \bigg( \frac{\m'}{r} - \frac{r}{2} \bigg) ((\partial_r r^{2\nu}) V) + r^{2\nu} \nabla_{\partial_r} V) \\
&= 2\nu r^{2\nu-1} \tr(dr \otimes \nabla V) + (\Delta r^{2\nu})V + 2\nu r^{2\nu-1} \tr(\nabla r \otimes V) + r^{2\nu} \Delta V \\ 
& \;\;\;\;\;\;\;\;\;\;\;\;\;\;\;\;+ \bigg( \frac{\m'}{r} - \frac{r}{2} \bigg) ((\partial_r r^{2\nu}) V) + r^{2\nu} \nabla_{\partial_r} V) \\
&= r^{2\nu} \mathcal{L}_{\m'} V + (\mathcal{L}_{\m'} r^{2\nu}) V + 4\nu r^{2\nu-1} \tr(dr \otimes \nabla V) \\
&= r^{2\nu} \mathcal{L}_{\m'} V + (\mathcal{L}_{\m'} r^{2\nu}) V + 4\nu r^{2\nu-1} \nabla_{\partial_r} V \\
&= r^{2\nu} \mathcal{L}_{\m'} V + (\mathcal{L}_{\m'} r^{2\nu}) V  + r^{2\nu} \frac{4\nu}{r} \nabla_{\partial_r} V\\
&= r^{2\nu} \mathcal{L}_{\m' + 4\nu}V + (\mathcal{L}_{\m'} r^{2\nu}) V.
\end{align*}
Note that a simple way to evaluate the trace $\tr(dr \otimes \nabla V)$ is to calculate in a coordinate system $(r, \mathbf{x})$ around $p \in \Sigma$, where $\mathbf{x}$ is a local coordinate system on $S_{r(p)}$. Finally, we observe that 
\begin{align*}
\mathcal{L}_{\m'} r^{2\nu} &=  -\nu r^{2\nu} + O(r^{-2+2\nu};\nu, \m').
\end{align*}
Plug this into the previous calculation to obtain
\[\mathcal{L}_{\m'} \hat{V} = r^{2\nu} \mathcal{L}_{\m' + 4\nu}V - \nu \hat{V} + \hat{V}O(r^{-2};\nu, \m').
\]
Finally, observe that $\nabla \hat{V} = r^{2\nu} \nabla V + 2\nu r^{2\nu -1} dr \otimes V$, and so
\[
r^{2\nu} |\nabla V| \le |\nabla \hat{V}| + \frac{4|\nu|}{r}r^{2\nu}|V| = |\nabla \hat{V}| + \frac{4|\nu|}{r}|\hat{V}|
\]
Thus,
\begin{align*}
|(\mathcal{L}_{\m-4\nu} + \lambda + \nu) \hat{V} | &= |r^{2\nu} \mathcal{L}_\m V - \nu \hat{V} + \hat{V}O(r^{-2};\nu, \m') + r^{2\nu} \lambda V + \nu \hat{V}|\\
&= r^{2\nu}|(\mathcal{L}_\m + \lambda)V| + |\hat{V}|O(r^{-2}; \nu, \m)\\
&\le Mr^{-2} (r^{2\nu}|V| + r^{2\nu}|\nabla V|) + |\hat{V}|O(r^{-2}; \nu, \m) \\
&\le M r^{-2} (|\hat{V}| + |\nabla \hat{V}| + \frac{4|\nu|}{r}|\hat{V}| ) + |\hat{V}|O(r^{-2}; \nu, \m) \\
& \le M'r^{-2}  (|\hat{V}| + |\nabla \hat{V}|)
\end{align*}
This concludes the proof of the proposition.
\end{proof}

\begin{proof}[Proof of Theorem \ref{mainthm4}] Theorem \ref{mainthm4} is an immediate consequence of Theorem \ref{mainthmharmver} and Proposition \ref{harm2eigen}. 
\end{proof}

\section{Unique Continuation on Self-Similar Ends}

We begin by proving that the ends of asymptotically conical high codimension self-shrinkers are weakly conical ends. We follow the method of proof in \cite[Lemma 8.1]{Ber} almost exactly, with a few modifications to extend the proof to the case of high codimension.

\begin{lem}\label{shrinkerconend}
Let $F:M^m\rightarrow \R^{n+m}$ be an asymptotically conical self-shrinker or self-expander. There is a radius $R_F$ so that if $g$ is the metric pulled back from $\R^{n+m}$ by $F$ to $M_{R_F}$ and $r(p) = |F(p)|$, then $(M_{R_F}, g, r)$ is a weakly conical end. 
\end{lem}

\begin{proof}
Lemma \ref{curvasymp} implies that if $R_F > R_1$, the second fundamental form $A$ of $M_{R_F}$ satisfies
\[
|A(p)| \le C_1 |F(p)|^{-1} = C_1 r(p)^{-1}. 
\]
The self-shrinker/self-expander equation implies that there is $\tilde{C}>C_1$ such that 
\[
|F^N(p)| = 2|\vec{H}(p)| \le C|A(p)| \le \tilde{C} r(p)^{-1} < \frac{1}{2}.
\]
Set $R_F$ large enough that $\tilde{C} R_F^{-1} + \tilde{C}^2 R_F^{-2} < \frac{1}{4}$. Since $|F^N(p)|<1/2$ for $|F(p)|> R_F$, 
\[
|F^T(p)| \ge 1/2|F(p)|
\]
for $p \in M_{R_F}$. We calculate
\[
||\nabla_g r(p)| - 1| = \bigg| \frac{|F^T(p)|}{|F(p)|} - 1 \bigg| = \frac{|F(p)|^2 - |F^T(p)|^2}{|F(p)|(|F^T(p)| + |F(p)|)} \le \frac{|F^N(p)|^2}{\frac{3}{2}|F(p)|^2} \le \tilde{C} r^{-4} < \frac{1}{2}
\]
Note that we use the self-shrinker/self-expander equation in the second to last inequality. Next, we confirm the estimate for the Hessian $\nabla^2_g r^2$. Observe that $\nabla r^2 = 2r \nabla r = 2 F^T$. Let $X, Y$ be two tangential vector fields with respect to the shrinker/expander $F$.
\begin{align*}
\nabla^2_g r^2 (X,Y) &= \nabla^g_X (\nabla_Y r^2) - \nabla^g_{\nabla_XY}r^2\\
&= \nabla^g_X g(Y, 2F^T) - g(\nabla^g_XY, 2F^T)\\
&= \nabla_X \langle Y, 2 F^T \rangle_{\R^{n+m}} - \langle \nabla^g_X Y , 2F^T \rangle_{\R^{n+m}} \\
&= \nabla_X \langle Y, 2 F \rangle_{\R^{n+m}} - \langle \nabla^g_X Y , 2F \rangle_{\R^{n+m}} \\
&= 2 \langle Y, \nabla_X F \rangle_{\R^{n+m}} +  \langle \nabla_X Y, 2 F \rangle_{\R^{n+m}} - \langle \nabla^g_X Y , 2F \rangle_{\R^{n+m}} \\
&= 2 \langle Y, \nabla_X F \rangle_{\R^{n+m}} +  \langle \nabla_X Y, - \nabla^g_X Y , 2F \rangle_{\R^{n+m}} \\
&= 2 g(X,Y) + 2 \langle A(X,Y), F \rangle_{\R^{n+m}}
\end{align*}
Thus,
\[
|\nabla^2_g r^2 - 2g| \le 2|F^N||A| \le \tilde{C}^2r^{-2} < \frac{1}{2}.
\]
We conclude that that $(M_{R_F},g,r)$ satisfies all the properties of a weakly conical end.
\end{proof}

\begin{lem}\label{lem-norm-bund-curv-decay}
The connection $\nabla^\perp$ on the normal bundle $NM_{1,R_4}$ satisfies Condition \ref{cond-curvature-decay}.
\end{lem}

\begin{proof}
Consider a point $x_0$ on $M_{1,R_4}$ and let $z_0$ be the nearest point to $x_0$ on the cone $C$. By Lemma \ref{graphest}, a neighborhood of $x_0$ in $M_{1,R_4}$ may be represented as the the graph of a function $u(x,1)$ defined on a disc $D^m_\rho$ centered at $z_0$ in the tangent space $T_{z_0} C$. Using the notation of Lemma \ref{graphest}, the function $u(x,0)$ represents a neighborhood of of $z_0$ in $C$ on the same disc $D^m_\rho$. By Corollary \ref{sepdecay}, the differences $|D^iu(z_0,1) - D^iu(z_0,0)| = O(|F_1(x_0)|^{-i-1})$. In particular, the equivalence of the Hessian of $u$ and the second fundamental form (c.f. Lemma \ref{HessBound}) implies that 
\begin{equation}\label{eq-difference-2nd-FF}
|A_C(p_*X,p_*Y) - A(X,Y)| \le O(|X||Y||F_1(x_0)|^{-3})
\end{equation}
when evaluated at a point $x_0$. Here the $A$ without subscript denotes the second fundamental form $A_{M_{1,R_4}}$ of $M_{1,R_4}$, and $X,Y$ are identified with $p_*X, p_*Y$ by pushforward along the projection map $p$ to the cone $C$. 

Given an $m$-dimensional cone $C \subset \R^{m+n}$, consider the radial vector field $\mathbf{x} = r\partial_r$ restricted to $C$. Let $Y$ be an arbitrary tangent vector field on $C$. We calculate the second fundamental form evaluated at $\mathbf{x}$ and $Y$. Let $\mathbf{x}$ denote the extension of the radial/coordinate vector field $x^\alpha \partial_\alpha$ to $\R^{n+m}$.
\begin{align*}
    A_C(Y, \mathbf{x}) &= (D^{\R^{m+n}}_Y \mathbf{x})^\perp \\
    &= (Y^\alpha \partial_\alpha (x^\beta \partial_\beta))^\perp \\
    &= (Y^\alpha \delta_\alpha^\beta \partial_\beta)^\perp \\
    &= (Y^\alpha \partial_\alpha)^\perp\\
    &= 0,
\end{align*}
since we assumed that $Y$ was tangential. 

Since $A_C(Y, \mathbf{x}) = 0$, it follows that $|A((p^{-1})_*Y, (p^{-1})_*\mathbf{x})| = O(|Y||F_1(x_0)|^{-2})$ by \eqref{eq-difference-2nd-FF}. We estimate the difference $|(p^{-1})_*\mathbf{x}(z_0) - F_1(x_0)^T|$, where we take the tangential part of $F_1(x_0)$ with respect to $TM_{1,R_4}$. Corollary \ref{sepdecay} implies that 
\[|C(z_0) - F_1(x_0)| , \big\langle \tfrac{\mathbf{x}(z_0)}{|\mathbf{x}(z_0)|}, \tfrac{\mathbf{x}(x_0)}{|\mathbf{x}(x_0)|} \big \rangle, \text{ and } \langle T_{z_0}C, T_{x_0} M_{1,R_4} \rangle = O(|F_1(x_0)|^{-1}).
\]
Thus, 
\begin{align*}
|A(Y, F_1(x_0)^T)| &\le |A(Y, (p^{-1})_*\mathbf{x})| + |A(Y, F_1(x_0)^T - (p^{-1})_*\mathbf{x})| \\
&\le |A(Y, (p^{-1})_*\mathbf{x})| + |A||Y| |F_1(x_0)^T - (p^{-1})_*\mathbf{x}| \\
& \le O(|Y||F_1(x_0)|^{-2}),
\end{align*}
where we have used the facts that $|A| = O(|F_1(x_0)|^{-1})$ (by Lemma \ref{curvasymp}) and that $|p_* Y| \simeq |Y|$ (by Lemma \ref{graphest}). 

Finally, we recall Ricci's equation \cite[(9)]{mcfhigher} applied to the ambient space $N = \R^{n+m}$. If $R^\perp$ is the curvature tensor of the connection $\nabla^\perp$ on the normal bundle $NM_{1,R_4}$, $X,Y \in \Gamma(TM_{1,R_4})$ and $W \in \Gamma(NM_{1,R_4}) $
\begin{equation}\label{eq-Ricci-equation}
    R^\perp (X,Y) W = - \sum_{i=1}^m \big( \langle W, A(X,e_i) \rangle A(Y,e_i) - \langle W, A(Y,e_i) \rangle A(X,e_i) \big),
\end{equation}
where $e_i$ is an arbitrary orthonormal frame. Setting $X = F_1(x_0)^T$ and using Lemma \ref{curvasymp} and the fact that $|A(Y, F_1(x_0)^T)| = O(|Y||F_1(x_0)|^{-2})$, we obtain
\begin{equation}
    |R^\perp (F_1(x_0)^T,Y) W| = O(|W||Y||F_1(x_0)|^{-3}),
\end{equation}
which is Condition \ref{cond-curvature-decay}.
\end{proof}

\begin{lem}\label{Vvanish} If a self-shrinking end $M_{2,K}$ can be represented as a normal section $V$ over a self-shrinking end $M_{1,R_4}$ which satisfies both the decay estimates \eqref{Vdecay} and 
\[\bigg|\bigg(\mathcal{L}_0 + \frac{1}{2}\bigg) V(p) \bigg| \le  M|F_1(p)|^{-2}(|V| + |\nabla^\perp V|), 
\]
then $V$ is identically 0 and $F_1$ and $F_2$ coincide.
\end{lem}

\begin{proof} By Lemma \ref{shrinkerconend}, $F_1$ is a weakly conical end and $V$ is an almost eigensection of $\mathcal{L}_0$ with eigenvalue $1/2$. We may apply Theorem \ref{mainthm4} to the almost eigensection $V$ and obtain that $V$ is asymptotically homogeneous of degree $2\lambda =1$, that $\tr^1_\infty V = a \in L^{2}(L(\Sigma); B|_{L(\Sigma)})$, and that
\begin{equation}\label{VL2bound}
\int_{\bar{E}_{R_0}} |V|^2 r^{-4-m} \le \frac{K_0}{R_0^2} ||a||^2_{L^2(L(\Sigma); B|_{L(\Sigma)})}.
\end{equation}
We claim that the leading term $A$ of $V$ is equal to 0. We test this by taking the $L^2_{loc}$ limit
\begin{align*}
\lim_{\tau \rightarrow \infty} \int_K |\Pi_\tau^* (r^{-1}V)|^2 d\mu_C &=  \lim_{\tau \rightarrow \infty} \int_{K} |(\tau r(p))^{-1} P_{\Pi_t(p), \tau} V(\Pi_\tau(p))|^2 d\mu_C \\
&= \lim_{\tau \rightarrow \infty} \int_K (\tau r(p))^{-2}|V(\Pi_\tau(p))|^2 d\mu_C\\
&\le \lim_{\tau \rightarrow \infty} C_3 \int_K (\tau r(p))^{-4} d\mu_C \\
& \le \lim_{\tau \rightarrow \infty} C_4 \tau^{-4} \mu_C(K) = 0.
\end{align*}
The first inequality comes from the bound on $V$ in Lemma \ref{secasymp}. Thus, the leading term $A$ is equal to 0. This implies that the trace at infinity of $V$ vanishes as well.
\[
a = \tr^1_\infty(V) = \tr^1(A) = 0.
\]
Plugging into inequality \eqref{VL2bound} implies that $V$ is uniformly equal to $0$ in the annular region $\bar{E}_{R_0}$. That $F_1$ and $F_2$ coincide in their region of definition follows from the analyticity of self-shrinkers. 
\end{proof}

\begin{cor}\label{Vvanishxpand} If a self-expanding end $M_{2,K}$ can be represented as a normal section $V$ over a self-expanding end $M_{1,R_4}$ which satisfies the estimates \eqref{Vdecay}, \eqref{Hausdorffdecay}, and 
\[\bigg|\bigg(\mathcal{L}_0^+ - \frac{1}{2}\bigg) V(p) \bigg| \le  M|F_1(p)|^{-2}(|V| + |F_1(p)|^{-1}|\nabla^\perp V|), 
\]
then $V$ is identically 0 and $F_1$ and $F_2$ coincide.
\end{cor}

\begin{proof}
By Lemma \ref{shrinkerconend}, $F_1$ is a weakly conical end and $V$ is an almost eigensection of $\mathcal{L}_0^+$ with eigenvalue $-1/2$. We may apply Theorem \ref{xpanderuniq} to the almost eigensection $V$ and obtain that $\hat{V} = \Psi_{m+1}V$ is asymptotically homogeneous of degree $0$, that $\tr^0_\infty \hat{V} = \hat{a} \in L^{2}(L(\Sigma); B|_{L(\Sigma)})$, and that
\begin{equation}\label{VhatL2bound}
\int_{\bar{E}_{R_0'}} |\hat{V}|^2 r^{-2-m} \le \frac{K_0}{R_0^2} ||\hat{a}||^2_{L^2(L(\Sigma); B|_{L(\Sigma)})}.
\end{equation}
Observe that 
\[\lim_{\rho \rightarrow \infty} \rho^{1-m}\int_{S_\rho} |\hat{V}|^2 = \int_{L(\Sigma)} |\hat{a}|^2.
\]
By \eqref{Hausdorffdecay} and the weakly conical end property, we calculate 
\begin{align*}\rho^{1-m}\int_{S_\rho} |\hat{V}|^2 &= \rho^{1-m}\Psi_{m+1}^2(\rho) \int_{S_\rho} |V|^2 \\
&\le \rho^{1-m} \Psi_{m+1}^2(\rho) K \rho^{m-1} \textrm{dist}_{\mathcal{H}}(F_1(M)\cap \partial B_\rho , F_2(M)\cap \partial B_\rho) \\
& \le \rho^{1-m}\Psi_{m+1}(\rho)^2 \rho^{-m-3} \Phi_0(\rho) o(1) = o(1)
\end{align*}
This implies that $||\hat{a}||_{L^2} = 0$, and thus that $\hat{V} = V = 0$ on $E_{R_0'}$. The analyticity of self-expanders allows us to say that $F_1$ and $F_2$ coincide where defined.
\end{proof}

\begin{proof}[Proof of Theorem \ref{maintheorem}] By Corollary \ref{nomult}, a connected AC shrinker end with multiplicity may be written as a section $V$ of its own normal bundle satisfying the decay estimates \eqref{Vdecay} and the hypotheses of Theorem \ref{mainthm4}. By Lemma \ref{Vvanish}, the section $V$ is identically 0 and the $k$ sheets of the $k$-fold covering space coincide. Applying this argument to each connected AC end implies that the image of a shrinker $F$ asymptotic to a cone $C$ may be written as a single valued normal section over the cone $C$ outside of a ball. This gives a reparametrization of $F(M) \setminus B_{R_4}$ with multiplicity 1.  

Let $F_1$ and $F_2$ be two shrinkers asymptotic to $C$ which have been reparametrized to have multiplicity 1 over $C$ outside some ball $B_R$. They are homeomorphic to $C_R$ via the covering projections $p_1$ and $p_2$. By Proposition \ref{normbund}, Lemma \ref{secasymp}, and Lemma \ref{linearization} outside some ball $F_2$ can be written as a section $V$ over $F_1$ satisfying the decay estimates \eqref{Vdecay} and the hypotheses of Theorem \ref{mainthm4}. By Lemma \ref{Vvanish}, $V\equiv 0$ so $F_1$ and $F_2$ coincide outside the ball $B_R$. The theorem follows from analytic unique continuation inside the ball. 
\end{proof}

\begin{proof}[Proof of Theorem \ref{xpandertheorem}] Corollary \ref{nomult} cannot be used to reduce the multiplicity of self-expanders, because the separation between the sheets of a single expander does not necessarily satisfy condition \eqref{Hausdorffdecay}. 

For the moment, let $F_1$ and $F_2$ refer to the restrictions of these expanding immersions to one of their connected AC ends. We consider the covers $\tilde F_1$ and $\tilde F_2$ corresponding to the subgroup $G = {p_1}_*(\pi_1(M_{1,R_3}, x_1)) \cap {p_2}_*(\pi_1(M_{2, R_3}, x_2))$, as in Remark \ref{univcov}. Proposition \ref{normbund}, Lemma \ref{secasymp}, and Corollary \ref{xpanderlinearization} suffice to establish that in some annular region, $\tilde M_{2,K}$ can be written as a normal section $V$ over $\tilde M_{1,R_4}$ which satisfies the hypotheses of Corollary \ref{Vvanishxpand} for expanders satisfying \eqref{Hausdorffdecay}. Therefore, $V$ vanishes and the images of $\tilde F_1$ and $\tilde F_2$ coincide. Recall that the induced immersion $\tilde F_i$ is $F_i \circ \mathscr{P}_i$ the composition of the original immersion $F_i:M_{i,K} \rightarrow \R^{n+m} $ and the covering projection $\mathscr{P}_i : \tilde M_{i,K} \rightarrow M_{i,K}$. Hence, the images of $F_1$ and $F_2$ coincide. Then apply this argument to each connected end to obtain the theorem. 
\end{proof}

\begin{rem}
If $F_1$ and $F_2$ are self-expanders satisfying the hypotheses of Theorem \ref{xpandertheorem}, then the coincidence of their images means they must converge to $C$ with the same ``true" multiplicity after a reparametrization. In fact, the conclusion of Theorem \ref{xpandertheorem} implies that one of the ends $M_{i,K}$ must evenly cover the other--thus, the desired reparametrization can be found by taking quotients with respect to deck transformations. In contrast to case of self-shrinkers, there may be many self-expanders which converge non-trivially to the cone $C$ with different multiplicities. 

Another subtlety to note is that if the link of the asymptotic cone is disconnected, the analysis is carried out on each end separately. Therefore, each connected expanding end may converge to its corresponding cone with different multiplicity. Nonetheless, even if this occurs, if $F$ and $\tilde F$ are complete AC expanders without boundary with the same asymptotic cone $C$ which satsify \eqref{Hausdorffdecay}, then Theorem \ref{xpandertheorem} combined with unique continuation inside the ball $B_{R_0}$ implies that they coincide.
\end{rem}

\appendix
\section{Geometry in High Codimension}

Here we collect a number of useful results that are used throughout the body of the paper. First, we recall a useful basic estimate on high codimension graphs over disks from the thesis of A. A. Cooper. 

\begin{lem}\label{HessBound}\cite[Lemma 2.1.2]{Coo} Let $f: D^m_r \rightarrow \R^n$ be a $C^2$ function on the disc of radius $r$. Then
\[ |D^2 f|^2 \le (1 + |Df|^2)^3 |A|_g^2, 
\]
where $|A|_g$ denotes the norm of the second fundamental from with respect to the immersion metric of the associated graph $x \in D^m_r \mapsto (x,f(x)) \in D^m_r \times \R^n$. 
\end{lem}

\begin{rem}The precise meaning of $|A|_g$ is as follows: Let $\{x_1, \ldots, x_m\}$ be coordinates on $D^m_r$ and let $\{y_1, \ldots, y_\alpha\}$ be coordinates on $\R^n$. Let
\[g_{ij} = (\vec{e}_i, D_{x_i} f) \cdot (\vec{e}_j, D_{x_j} f)
\]
and 
\[g_{\alpha \beta} = (-Df_\alpha, \vec{e}_\alpha) \cdot (-Df_\beta, \vec{e}_\beta)
\]
be metrics on the tangent and normal bundle respectively. If $g^{ij}$ and $g^{\alpha \beta}$ denote the inverse matrices to $g_{ij}$ and $g_{\alpha \beta}$ respectively, then the the norm-squared $|A|_g^2$ is 
\[|A|_g^2 = \frac{\partial^2 f_\alpha}{\partial x_i \partial x_j}\frac{\partial^2 f_\beta}{\partial x_k \partial x_l}g^{\alpha\beta}g^{ik}g^{jl}.
\]
\end{rem}

We restate the extension of this result for higher order terms, with an explicit bound for the third derivative.

\begin{lem}\label{HOT} \cite[Lemma 2.1.3]{Coo}
For any $\ell \ge 2$, we can bound $|D^\ell f|$ in terms of $|Df|$, $|D^2f|$,..., $|D^{\ell-1}f|$, $|\nabla^{\ell-2} A|_g$, and absolute constants depending on $m, n,$ and $\ell$. In particular, for $\ell =3 $,
\[|D^3 f| \le (1+ |Df|^2)^2 |\nabla A|_g + \big(2\sqrt{2m + 4\sqrt{mn} +n} \big)|D^2f|^2 |Df|
\]
\end{lem}

\section{Results for Self-Expanders}

In this appendix, we collect a number of results from \cite{Ber} that allow us to extend the arguments used to prove Theorem \ref{maintheorem} for self-shrinkers to the case of self-expanders in Theorem \ref{xpandertheorem}. To extend the proofs of these results for scalar functions to vector bundles over weakly conical ends is in general straightforward after either replacing the scalar functions $|u|$ and $|\nabla_g u|$ with $|V|$ and $|\nabla V|$ or imitating calculations performed in Section 4. Therefore, we omit proofs. 

The operator $\mathcal{L}_0^+$ appears in Corollary \ref{xpanderlinearization}. It can be seen as one of a class of operators 
\[ \mathcal{L}^+_\m = \Delta + \frac{r}{2}\nabla_{\partial_r} + \frac{\m}{r} \nabla_{\partial_r} 
\]
associated to the weights
\[\Psi_\mu = r^\mu e^{\frac{r^2}{4}}
\]
Analogous to Proposition \ref{harm2eigen}, the almost eigensections of $\mathcal{L}^+_\m$ can be transformed into almost eigensections of an associated $\mathcal{L}_{\m '}$, for some $\m'$. 

\begin{prop}\label{plustominus} There is a constant $M' = M'(M, \m, \nu, n, \Lambda)$ such that if $(\Sigma^m, g, r)$ is an asymptotically conical end with associated constant $\Lambda$, $(B,p,\Sigma, h, \nabla)$ is a vector bundle of rank $n$ over $\Sigma$ with metric $h$, and $V \in C^2(\Sigma; B)$ satisfies
\begin{enumerate}
    \item $|(\mathcal{L}_\m + \lambda )V| \le Mr^{-2}(|V| + |\nabla V|)$, then $\hat{V} = \Phi_\mu V$ satisfies
    \[\bigg| \bigg( \mathcal{L}^+_{\mu - 2\nu} + \frac{1}{2}(n + m + 2\lambda - \nu)\bigg) \hat{V} \bigg| \le M'r^{-1}(|\hat{V}| + r^{-1}|\nabla \hat{V}|)
    \]
    \item $|(\mathcal{L}^+_\mu + \lambda) V| \le M r^{-2} (|V| + r^{-1} |\nabla V|)$, then $\hat{V} = \Psi_\mu(r)V$
     \[\bigg| \bigg( \mathcal{L}_{\mu - 2\nu} + \frac{1}{2}(-n - m + 2\lambda + \nu)\bigg) \hat{V} \bigg| \le M'r^{-2}(|\hat{V}| + |\nabla \hat{V}|)
    \]
\end{enumerate}
\end{prop}

The proof is almost identical to that of \cite[Proposition 7.1]{Ber}, and all necessary modifications for the vector bundle case can be found in Proposition \ref{harm2eigen}. Therefore, we may apply Theorem \ref{mainthmharmver} to obtain the following theorem analogous to \cite[Theorem 7.2]{Ber}.

\begin{thm}\label{xpanderuniq} If $\hat{V} = \Psi_{m-2\lambda}V \in C^2_{4\lambda -2m +2,1}(\Sigma; B)$ and $V$ satisfies
\[ |(\mathcal{L}^+_0 + \lambda) V| \le M r^{-2} (|V| + r^{-1}|\nabla V|) 
\]
then there are constants $R_0'$ and $K_0'$, depending on $V$, so that for any $R \ge R_0'$
\[
\int_{\bar{E}_R} \bigg( |\hat{V}|^2 +r^2|\nabla \hat{V}|^2 + r^4 | \nabla_{\partial_r} \hat{V}| \bigg) r^{-1-m} \le \frac{K_0'}{R^{m}} \int_{S_R} |\hat{V}|^2
\]
Moreover, $\hat{V}$ is asymptotically homogeneous of degree $0$ and $\tr_\infty^{0} \hat{V} = \hat{a}$ for some section $\hat{a} \in L^2(L(\Sigma); B|_{L(\Sigma)})$ that satisfies $\alpha^2 = \lim_{\rho \rightarrow \infty} \rho^{1-m} \int_{S_\rho} |\hat{V}|^2 = \int_{L(\Sigma)} |\hat{a}|^2$ and, 
\[
\int_{\bar{E}_R} \bigg( |\hat{V}|^2 + r^2 (|\hat{V}- \hat{A}|^2 + |\nabla \hat{V}|^2) + r^4 |\nabla_{\partial_r} V|^2| \bigg) r^{-2-m}  \le \frac{K_0' \alpha^2}{R^2}.
\]
Here $\hat{A} \in L^2_{loc}(\Sigma; B)$ is the leading term of $\hat{V}$ and $L(\Sigma)$ is the link of the asymptotic cone. 
\end{thm}

The final result (analogous to \cite[Theorem 9.1]{Ber}) needed to prove Theorem \ref{xpandertheorem} tells us that the normal section $V$ representing an expander satisfies the hypothesis $\Psi_{m-2\lambda}V \in C^2_{4\lambda -2m +2,1}(\Sigma; B)$ from Theorem \ref{xpanderuniq}. The arguments of Section 9 in \cite{Ber} do not substantially change in the vector-valued case. The necessary modifications to the proofs can be obtained by imitating the proof of the Poincar\'{e} inequality in Proposition \ref{poincareineq} and recalling the integration by parts formula used in the derivation of identity \eqref{dirichlet}. 

\begin{thm}\label{xpanddecay}
If $V \in C^2(\bar{E}_R;B)$ satisfies 
\[ |(\mathcal{L}_0^+ + \lambda )V| \le Mr^{-1}(|V| + |\nabla V|) \text{ and } B(\rho) = o( \rho^{-4\lambda +m -1}), \rho \rightarrow \infty,
\]
then $\Psi_0 V \in C^2_{\m', 1}(\bar{E}_R; B)$ for any $\m'$.
\end{thm}

\section*{Acknowledgments} 
I would like to thank: Bing Wang, for introducing me to the article \cite{Wang} and for many helpful discussions; Lu Wang, for pointing me towards the article \cite{Ber} and for illuminating comments on unique continuation; Mariel S\'{a}ez, for a discussion on issues of orientability that motivated the addition of Remark \ref{rem-orientability}; Sigurd Angenent, for several suggestions that became the genesis of Lemma 3.8 and of the proof of Lemma 3.3, and for carefully reading reading several sections of this article. I would also like to thank the anonymous journal referee(s) for important comments and corrections.

\end{document}